\documentclass{amsart}    
\usepackage{amsfonts}
\usepackage{amssymb}
\usepackage{latexsym}
\usepackage{graphicx}
\usepackage[final]{hyperref}
\hypersetup{colorlinks=true, linkcolor=blue, anchorcolor=blue, citecolor=red, filecolor=blue, menucolor=blue, pagecolor=blue, urlcolor=blue}

\renewcommand{\theenumi}{\roman{enumi}}

\newcommand{\nullform}{\mathfrak B_{\theta_{12}}}
\newcommand{\truenullform}{\mathfrak B_{\theta_{12} \ll 1}}
\newcommand{\nullformhalf}{\mathfrak B_{\sqrt{\theta_{12}}}}

\newcommand{\Nmin}{N_{\mathrm{min}}}
\newcommand{\Nmed}{N_{\mathrm{med}}}
\newcommand{\Nmax}{N_{\mathrm{max}}}
\newcommand{\Lmin}{L_{\mathrm{min}}}
\newcommand{\Lmed}{L_{\mathrm{med}}}
\newcommand{\Lmax}{L_{\mathrm{max}}}

\newcommand{\hypwt}{\mathfrak h}

\newcommand{\Abs}[1]{\left\vert #1 \right\vert}
\newcommand{\abs}[1]{| #1 |}
\newcommand{\bigabs}[1]{\bigl\vert #1 \bigr\vert}

\newcommand{\norm}[1]{\left\Vert #1 \right\Vert}

\newcommand{\bignorm}[1]{\bigl\Vert #1 \bigr\Vert}
\newcommand{\Bignorm}[1]{\Bigl\Vert #1 \Bigr\Vert}

\newcommand{\N}{\mathbb{N}}
\newcommand{\Proj}{\mathbf{P}}

\newcommand{\R}{\mathbb{R}}
\newcommand{\Z}{\mathbb{Z}}

\DeclareMathOperator{\supp}{supp}

\newtheorem{theorem}{Theorem}[section]
\newtheorem{lemma}{Lemma}[section]
\newtheorem{corollary}{Corollary}[section]

\theoremstyle{definition}
\newtheorem{definition}{Definition}[section]
\newtheorem{example}{Example}[section]

\theoremstyle{remark}
\newtheorem{remark}{Remark}[section]

\title[Anisotropic bilinear estimates]{Anisotropic bilinear $L^2$ estimates related to the 3d wave equation}

\author{Sigmund Selberg}
\address{Department of Mathematical Sciences\\ Norwegian University of Science and Technology\\ Alfred Getz' vei 1\\ N-7491 Trondheim\\ Norway}
\email{sselberg@math.ntnu.no}
\urladdr{www.math.ntnu.no/~sselberg}
\thanks{Partially supported by Research Council of Norway, grant no.\ 160192/V30}
\thanks{The author wishes to thank Sergiu Klainerman for his hospitality and encouragement.}

\begin{document}

\begin{abstract}
We first review the $L^2$ bilinear generalizations of the $L^4$ estimate of Strichartz for solutions of the homogeneous 3D wave equation, and give a short proof based solely on an estimate for the volume of intersection of two thickened spheres. We then go on to prove a number of new results, the main theme being how additional, anisotropic Fourier restrictions influence the estimates. Moreover, we prove some refinements which are able to simultaneously detect both concentrations and nonconcentrations in Fourier space.
\end{abstract}

\maketitle

\tableofcontents


\section{Introduction}\label{A}

In this paper we are interested in bilinear $L^2$ Fourier restriction estimates related to solutions of the homogeneous 3D wave equation
\begin{equation}\label{A:2}
  \square u = 0 \qquad \left(t \in \R, \;x \in \R^3; \; \square = -\partial_t^2 + \Delta\right),
\end{equation}
where $\Delta$ is the Laplacian on $\R^3_x$.

In general, by a \emph{bilinear $L^2$ Fourier restriction estimate on $\R^n$}, for a given $n \ge 1$, we mean here an estimate of the form
\begin{equation}\label{A:10}
  \norm{\Proj_{A_0}(\Proj_{A_1} f_1 \cdot \Proj_{A_2} f_2)}
  \le C_{A_0,A_1,A_2} \norm{\Proj_{A_1} f_1} \norm{\Proj_{A_2} f_2}
  \qquad (\forall f_1,f_2 \in \mathcal S(\R^n)),
\end{equation}
for given measurable sets $A_0,A_1,A_2 \subset \R^n$. Here $\mathcal S(\R^n)$ is the Schwartz space, $\norm{\cdot}$ denotes the norm on $L^2(\R^n)$, and $\Proj_{A}$, for any measurable set $A \subset \R^n$, denotes the multiplier whose symbol is the characteristic function $\chi_{A}$. That is,
\begin{equation}\label{A:11}
  \widehat{\Proj_A f} = \chi_A \widehat f,
\end{equation}
where
$$
  \widehat f(\xi) = \mathcal F f(\xi) = \int_{\R^n} e^{-i x \cdot \xi} f(x) \, dx
$$
is the Fourier transform. Of course, \eqref{A:10} is only interesting when the $A_j$ have nonzero $n$-dimensional volume, but restriction to hypersurfaces can be treated by thickening them slightly, as examples given below demonstrate.

If the set $A$ in \eqref{A:11} is given by some condition, we occasionally just replace the subscript $A$ in \eqref{A:11} by that condition, to avoid having to give a name to the set.

In our applications, $n=3$ or $1+3$, and in the latter case we split the Fourier variable into $(\tau,\xi)$, where $\tau \in \R$, $\xi \in \R^3$ are the Fourier variables of $t,x$ respectively. We shall call $\xi$ the spatial frequency. The characteristic set of \eqref{A:2} is the null cone $\abs{\tau}=\abs{\xi}$. We recall, however, that solutions of \eqref{A:2} split naturally into $u=u_+ + u_-$, where the $u_\pm$ satisfy
\begin{equation}\label{A:4}
  \left( i \partial_t \pm \abs{D} \right) u_\pm = 0,
\end{equation}
$\abs{D}$ being the multiplier with symbol $\abs{\xi}$; the corresponding characteristic sets are the null cone components
$$
  K^\pm = \left\{ (\tau,\xi) \in \R^{1+3} \colon \tau=\pm\abs{\xi} \right\}.
$$
For $L > 0$ we define also the thickened cones
$$
  K^\pm_L = \left\{ (\tau,\xi) \in \R^{1+3} \colon \bigabs{-\tau\pm\abs{\xi}} \le L \right\},
$$
which arise if we consider estimates of the form \eqref{A:10} related to solutions of \eqref{A:2}. Such estimates often depend on the size of the spatial frequency; to describe such restrictions we introduce notation for balls and annuli in $\R^3$:
$$
  B_N = \left\{ \xi \in \R^3 \colon \abs{\xi} \le N \right\},
  \qquad
  \Delta B_N = B_{N} \setminus B_{\frac{N}{2}} = \left\{ \xi \in \R^3 \colon \frac{N}{2} < \abs{\xi} \le N \right\},
$$
where $N > 0$. Later we introduce various anisotropic restrictions on the spatial frequency, which can also affect the estimates.

We now give some examples of well-known Fourier restriction theorems which fit into the form \eqref{A:10} (details are given below):

\begin{enumerate}
\item\label{A:20} The Stein-Tomas restriction theorem for the sphere $\mathbb S^2 \subset \R^3$. See \cite{Stein:1993}.
\item\label{A:22} Strichartz's $L^4$ estimate for the homogeneous 3d wave equation. See \cite{Strichartz:1977}.
\item\label{A:24} $L^2$ bilinear generalizations of Strichartz's estimate. See \cite{Klainerman:1993, Klainerman:1996,Foschi:2000}.
\end{enumerate}

Our main interest is in reviewing and extending \eqref{A:24}, but \eqref{A:20} and \eqref{A:22} provide some motivation for our point of view. In section \ref{B} we briefly review a general method for proving estimates of the form \eqref{A:10} (see \cite{Tao:2001} for a more wide-ranging discussion of this theme), and in section \ref{C} we implement this to give short, unified proofs of \eqref{A:20}--\eqref{A:24}, based solely on an elementary volume estimate for the intersection of two thickened spheres (see Lemma \ref{C:Lemma1} below).

Our main goal, however, is to prove a number of new results, presented in section \ref{L:0}, which turn out to be important in the study of regularity for the Maxwell-Dirac system; see \cite{Selberg:2008b}. The main theme is how additional, anisotropic Fourier restrictions influence the estimates. Moreover, we prove some refinements which are able to simultaneously detect both concentrations and nonconcentrations in Fourier space. To prove these estimates, we rely mostly on angular decompositions, orthogonality arguments, and volume estimates for intersections of sets with transversality in some direction. Our general approach here is very much in the spirit of the article \cite{Tao:2001}, which was one of the main sources of inspiration for the present work.

In estimates we use the shorthand $X \lesssim Y$ for $X \le CY$, where $C \gg 1$ is some absolute constant; $X=O(R)$ is short for $\abs{X} \lesssim R$; $X \sim Y$ means $X \lesssim Y \lesssim X$; $X \ll Y$ stands for $X \le C^{-1} Y$, with $C$ as above. We write $\simeq$ for equality up to multiplication by an absolute constant (typically factors involving $2\pi$). 

Let us first show how \eqref{A:20}--\eqref{A:24} fit into the framework \eqref{A:10}.

\begin{example}\label{A:30} The Stein-Tomas theorem for $\mathbb S^2 \subset \R^3$. The endpoint case of this theorem reads, in the dual formulation,
\begin{equation}\label{A:34}
  \norm{Tg}_{L^4(\R^3)} \lesssim \norm{g}_{L^2(\mathbb S^2,d\sigma)}
  \qquad (\forall g \in \mathcal S(\R^3)),
\end{equation}
where $d\sigma$ is surface measure on $\mathbb S^2$ and
$$
  \widehat{Tg}(\xi) = g(\xi) \, d\sigma(\xi) = g(\xi) \delta(1-\abs{\xi}).
$$
Approximating the point mass $\delta$ by $(1/2\varepsilon)\chi_{(-\varepsilon,\varepsilon)}$ with $\varepsilon > 0$ tending to zero, it is then easy to see that \eqref{A:34} follows from the estimate
\begin{equation}\label{A:36}
  \norm{\Proj_{S_\varepsilon} g}_{L^4(\R^3)} \lesssim \varepsilon^{1/2}
  \norm{\Proj_{S_\varepsilon} g}_{L^2(\R^3)},
  \qquad
  \text{where $S_\varepsilon = \left\{ \xi \in \R^3 \colon \bigabs{\abs{\xi} - 1} \le\varepsilon \right\}$}.
\end{equation}
Since \eqref{A:36} is an $L^4$ estimate, it can be restated as a bilinear $L^2$ estimate:
\begin{equation}\label{A:40}
  \norm{\Proj_{S_\varepsilon} g_1 \cdot \Proj_{S_\varepsilon} g_2}_{L^2(\R^3)} 
  \lesssim \varepsilon
  \norm{\Proj_{S_\varepsilon} g_1}_{L^2(\R^3)}
  \norm{\Proj_{S_\varepsilon} g_2}_{L^2(\R^3)},
\end{equation}
which is of the form \eqref{A:10}.
\end{example}

\begin{example}\label{A:50} Strichartz's $L^4$ estimate. In frequency-localized form, this says that for any $N > 0$,
\begin{equation}\label{A:52}
  \norm{T_\pm g}_{L^4(\R^{1+3})} \lesssim N^{1/2} \norm{g}_{L^2(\R^3)}
  \quad
  \left( \text{$\forall g \in \mathcal S(\R^3)$ s.t.\ $\supp g \subset \Delta B_N$}\right),
\end{equation}
where
$$
  \widehat{T_\pm g}(\tau,\xi) = g(\xi) \delta(-\tau\pm\abs{\xi}).
$$
Thus, $T_\pm g$ is the solution of \eqref{A:4} with data $g$ at $t=0$. Note that $\delta(-\tau\pm\abs{\xi})$ is, up to a constant factor, surface measure on the cone $K^\pm$. Approximating the point mass $\delta$ as above, one can show that \eqref{A:52} follows from the estimate, for any $N,L > 0$,
\begin{equation}\label{A:58}
  \Bignorm{\Proj_{(\R \times \Delta B_N) \cap K^\pm_L} u}_{L^4(\R^{1+3})} \lesssim (NL)^{1/2} 
  \Bignorm{\Proj_{(\R \times \Delta B_N) \cap K^\pm_L} u}_{L^2(\R^{1+3})}.
\end{equation}
\end{example}

\begin{example}\label{A:70} Bilinear generalizations of \eqref{A:52} were first investigated by Klainerman and Machedon \cite{Klainerman:1993, Klainerman:1996}. The following frequency-localized estimates were proved in \cite{Foschi:2000}, and also in \cite{Tataru:2001} and \cite{Tao:2001}; see \cite{Klainerman:2001} for a different approach to proving bilinear estimates. Given $N_0, N_1, N_2 > 0$, write
\begin{equation}\label{A:71}
  \Nmin^{12} = \min(N_1,N_2), \qquad \Nmin^{012} = \min(N_0,N_1,N_2).
\end{equation}
Let $\pm_1,\pm_2$ be arbitrary signs. Then for $g_1,g_2 \in \mathcal S(\R^3)$ with $\supp g_j \subset B_{N_j}$, $j=1,2$,
\begin{equation}\label{A:72}
  \norm{\Proj_{\R \times B_{N_0}} \left( T_{\pm_1} g_1
  \cdot T_{\pm_2} g_2 \right)}_{L^2(\R^{1+3})}
  \lesssim \left(\Nmin^{012}\Nmin^{12}\right)^{1/2}
  \norm{g_1}_{L^2(\R^3)} \norm{g_2}_{L^2(\R^3)}.
\end{equation}
By approximation, one can show that \eqref{A:72} follows from (in fact, is equivalent to)
\begin{multline}\label{A:76}
  \Bignorm{\Proj_{\R \times B_{N_0}} \Bigl( \Proj_{(\R \times B_{N_1}) \cap K^{\pm_1}_{L_1}} u_1
  \cdot \Proj_{(\R \times B_{N_2}) \cap K^{\pm_2}_{L_2}} u_2 \Bigr)}  \\
  \lesssim 
  \left( \Nmin^{012} \Nmin^{12} L_1 L_2 \right)^{1/2}
  \Bignorm{\Proj_{(\R \times B_{N_1}) \cap K^{\pm_1}_{L_1}} u_1}
  \Bignorm{\Proj_{(\R \times B_{N_2}) \cap K^{\pm_2}_{L_2}} u_2},
\end{multline}
where $L_1,L_2 > 0$ are arbitrary and $\norm{\cdot}$ denotes the norm on $L^2(\R^{1+3})$.
\end{example}

\begin{remark}\label{A:80} Due to the factor $(\Nmin^{012})^{1/2}$ on the right, it suffices to prove \eqref{A:76} with the balls $B_{N_j}$ replaced by the annuli $\Delta B_{N_j}$, for $j=0,1,2$; see section \ref{C}. From now on we shall generally use annuli instead of balls, and to simplify we write
\begin{equation}\label{A:77}
  K^{\pm}_{N,L}
  = \left( \R \times \Delta B_N \right)
  \cap K^{\pm}_{L}.
\end{equation}
\end{remark}

\begin{remark}\label{A:82} If $\pm_1=\pm_2$, then the factor $\Nmin^{12}$ inside the parentheses on the right hand side of \eqref{A:72} can be replaced by $\Nmin^{012}$; this fact is rarely useful in practice, however, so we shall ignore it.
\end{remark}

This concludes the discussion of how the restriction theorems \eqref{A:20}--\eqref{A:24} fit into the framework \eqref{A:10}. In section \ref{C} we prove these theorems in a unified manner, using only the following estimate for the volume of intersection of two thickened spheres. Introducing the notation
$$
  S_\delta(r) = \left\{ \xi \in \R^3 \colon r-\delta \le \abs{\xi} \le r+\delta \right\},
$$
we have:

\begin{lemma}\label{C:Lemma1}
Let $0 < \delta \ll r$ and $0 < \Delta \ll R$. Then for any $\xi_0 \in \R^3$,
$$
  \Abs{S_\delta(r) \cap ( \xi_0 + S_\Delta(R))}
  \lesssim \frac{rR\delta\Delta }{\abs{\xi_0}}.
$$
\end{lemma}

The proof is given in section \ref{G}.

Our new results are presented in the next section, but in preparation for this we introduce some more notation and terminology.

Throughout, $N_0,N_1,N_2,L_0,L_1,L_2 > 0$; $\pm_0,\pm_1,\pm_2$ denote arbitrary signs; we assume that
$u_1,u_2 \in L^2(\R^{1+3})$ satisfy
\begin{equation}\label{A:120}
  \supp \widehat{u_j} \subset K^{\pm_j}_{N_j,L_j} \qquad \text{for $j=1,2$},
\end{equation}
with notation as in \eqref{A:77}. Given $\gamma > 0$ and $\omega \in \mathbb S^2$, we also define $u_j^{\gamma,\omega}$ by
\begin{equation}\label{A:120:2}
  \supp \widehat{u_j^{\gamma,\omega}} \subset K^{\pm_j}_{N_j,L_j,\gamma,\omega} \qquad \text{for $j=1,2$},
\end{equation}
where
\begin{equation}\label{D:42}
  K^{\pm}_{N,L,\gamma,\omega}
  = \left\{ (\tau,\xi) \in K^{\pm}_{N,L} \colon \theta(\pm\xi,\omega) \le \gamma \right\},
\end{equation}
and $\theta(a,b)$ denotes the angle between nonzero $a,b \in \R^3$. Except in section \ref{C}, $\norm{\cdot}$ denotes the norm on $L^2(\R^{1+3})$. The shorthand \eqref{A:71} is used for both $N$'s and $L$'s, and analogous notation is used for maxima. In the case of a three-index such as $012$, $\Nmed^{012}$ denotes the median. 

For later use we note the following restatement of \eqref{A:76} in a more symmetric form, permitting the use of duality (that is, permutation).

\begin{theorem}\label{A:Thm1} For all $u_1,u_2 \in L^2(\R^{1+3})$ satisfying \eqref{A:120}, the estimate
\begin{equation}\label{A:100}
  \Bignorm{ \Proj_{K^{\pm_0}_{N_0,L_0}}
  ( u_1 u_2 ) }
  \le C
  \norm{u_1}
  \norm{u_2}
\end{equation}
holds with
\begin{align}
\label{A:110}
  C^2
  &\sim \Nmin^{012}\Nmin^{12} L_1 L_2,
  \\
  \label{A:112}
  C^2
  &\sim \Nmin^{012}\Nmin^{0j} L_0 L_j \qquad (j=1,2),
  \\
  \label{A:114}
  C^2
  &\sim N_0\Nmin^{12} \Lmin^{012} \Lmed^{012},
\end{align}
for any choice of signs $(\pm_0,\pm_1,\pm_2)$.
\end{theorem}

\begin{proof}
First, \eqref{A:110} follows from (in fact, is equivalent to, in view of Remark \ref{A:80}) the estimate \eqref{A:76}, and then \eqref{A:112} follows by permutation, i.e., by duality; see the general discussion in section \ref{B}; the duality argument works because the signs are arbitrary. From \eqref{A:140} below we see that the left hand side of \eqref{A:100} vanishes unless $\Nmin^{012}\Nmax^{012} \sim N_0\Nmin^{12}$. Therefore, combining \eqref{A:110} and \eqref{A:112}, we get \eqref{A:114}.
\end{proof}

Note the convolution formula
\begin{equation}\label{A:130}
  \widehat{u_1 u_2}(X_0)
  \simeq
  \int \widehat{u_1}(X_1)\widehat{u_2}(X_2)  
  \, \delta(X_0-X_1-X_2) \, dX_1 \, dX_2,
\end{equation}
where $X_j = (\tau_j,\xi_j) \in \R^{1+3}$, $j=0,1,2$. Thus, in \eqref{A:130},
\begin{equation}\label{A:138}
  X_0 = X_1 + X_2 \qquad (\iff \tau_0=\tau_1+\tau_2, \; \xi_0=\xi_1+\xi_2).
\end{equation}

\begin{definition}\label{A:Def}
A triple $(X_0,X_1,X_2)$ of vectors $X_j = (\tau_j,\xi_j) \in \R^{1+3}$ is said to form a \emph{bilinear interaction} if \eqref{A:138} holds.
\end{definition}

Since $\xi_0=\xi_1+\xi_2$ in a bilinear interaction, $\abs{\xi_j} \le \abs{\xi_k} + \abs{\xi_l}$ for all permutations $(j,k,l)$ of $(0,1,2)$, hence one of the following must hold:
\begin{subequations}\label{A:140}
\begin{alignat}{2}
  \label{A:140a}
  \abs{\xi_0} &\ll \abs{\xi_1} \sim \abs{\xi_2}& \qquad &(\text{``low output''}),
  \\
  \label{A:140b}
  \abs{\xi_0} &\sim \max(\abs{\xi_1},\abs{\xi_2})& \qquad &(\text{``high output''}).
\end{alignat}
\end{subequations}
The integration in \eqref{A:130} may be restricted to the region where $\xi_1,\xi_2 \neq 0$, hence
\begin{equation}\label{A:170}
  \theta_{12} \equiv \theta(\pm_1\xi_1,\pm_2\xi_2)
\end{equation}
is well-defined. Given signs $\pm_0,\pm_1,\pm_2$, we define also
\begin{equation}\label{N:12}
  \hypwt_j =  -\tau_j\pm_j\abs{\xi_j} \qquad (j=0,1,2),
\end{equation}
which we call \emph{hyperbolic weights}, whereas the $\abs{\xi_j}$ are called \emph{elliptic weights}.

For nonzero $a,b \in \R^3$, $\theta(a,b)$ denotes the angle between $a,b$. We note that
\begin{align}
  \label{I:2}
  \abs{a}+\abs{b}-\abs{a+b} &\sim \min(\abs{a},\abs{b}) \theta(a,b)^2,
  \\
  \label{I:4}
  \abs{a-b}-\bigabs{\abs{a}-\abs{b}} &\sim \frac{\abs{a}\abs{b}}{\abs{a-b}} \theta(a,b)^2 \qquad(a \neq b),
\end{align}
due to the identities
\begin{align*}
  \abs{a}+\abs{b}-\abs{a+b}
  &= \frac{(\abs{a}+\abs{b})^2-\abs{a+b}^2}{\abs{a}+\abs{b}+\abs{a+b}}
  = \frac{2\abs{a}\abs{b}\bigl(1-\cos\theta(a,b)\bigr)}{\abs{a}+\abs{b}+\abs{a+b}},
  \\
  \abs{a-b}-\bigabs{\abs{a}-\abs{b}}
  &= \frac{\abs{a-b}^2-\bigabs{\abs{a}-\abs{b}}^2}{\abs{a-b}+\bigabs{\abs{a}-\abs{b}}}
  = \frac{2\abs{a}\abs{b}\bigl(1-\cos\theta(a,b)\bigr)}{\abs{a-b}+\bigabs{\abs{a}-\abs{b}}},
\end{align*}
and the fact that $1-\cos\theta \sim \theta^2$ for $\theta \in [0,\pi]$.

\section{Main results}\label{L:0}

Here $N_0,N_1,N_2,L_0,L_1,L_2 > 0$, $\pm_0,\pm_1,\pm_2$ denote arbitrary signs, we assume that $u_1,u_2 \in L^2(\R^{1+3})$ satisfy \eqref{A:120}, and $\norm{\cdot}$ denotes the norm on $L^2(\R^{1+3})$.

\subsection{Anisotropic bilinear estimate}\label{X}

One of the key questions motivating the present work is to what extent additional Fourier restrictions lead to improvements in the standard bilinear estimates of Theorem \ref{A:Thm1}.

For example, if we start with the standard estimate
\begin{equation}\label{X:18}
  \norm{u_1u_2} \lesssim \left((\Nmin^{12})^2L_1L_2\right)^{1/2}
  \norm{u_1}\norm{u_2},
\end{equation}
and then restrict the spatial output frequency $\xi_0$ to a ball $B$ of radius $r \ll \Nmin^{12}$ and arbitrary center, the estimate improves to\footnote{This is more general than \eqref{A:110} (which corresponds to the special case $\xi_* = 0$ and $r=N_0$) in the low output case $N_0 \ll N_1 \sim N_2$, since \eqref{X:20} tells us that the position of the ball is irrelevant.}
\begin{equation}\label{X:20}
  \norm{\Proj_{\R \times B}(u_1u_2)} \lesssim \left(r\Nmin^{12}L_1L_2\right)^{1/2}
  \norm{u_1}\norm{u_2}.
\end{equation}
Moreover, since the position of the ball is arbitrary, $\Proj_{\R \times B}$ can equivalently be placed in front of either $u_1$ or $u_2$, as can be seen by a standard tiling argument, essentially as in the proof of Lemma \ref{B:Lemma2} below.

The estimate \eqref{X:20} is easily proved by modifying the proof that we give for \eqref{A:76} (in section \ref{C}), but it also follows immediately from the following much more powerful anisotropic estimate, where instead of restricting to a ball we just restrict the spatial frequency in a single direction $\omega \in \mathbb S^2$. To be precise, we restrict to a thickened plane given by $\xi\cdot\omega \in I$, where $I$ is an interval.

\begin{theorem}\label{Z:Thm1}
Let $\omega \in \mathbb S^2$, and let $I \subset \R$ be a compact interval. Assume $\widehat{u_1}$ is supported outside an angular neighborhood of the orthogonal complement $\omega^\perp$ of $\omega$:
\begin{equation}\label{Z:10}
  \supp \widehat{u_1} \subset
  \left\{ (\tau,\xi) \colon
  \theta(\xi,\omega^\perp) \ge \alpha \right\}
  \qquad \text{for some $\;0 < \alpha \ll 1$}.
\end{equation}
Assuming also \eqref{A:120} as usual, we then have
\begin{equation}\label{Z:14}
  \norm{\Proj_{\xi_0 \cdot \omega \in I}
  (u_1 u_2)}
  \lesssim \left( \frac{\abs{I}\Nmin^{12}L_1L_2 }{\alpha} \right)^{1/2}
  \norm{u_1}
  \norm{u_2},
\end{equation}
where $\abs{I}$ is the length of $I$.
\end{theorem}

We remark that since the position of the interval $I$ is irrelevant, the restriction can also be put, equivalently, on either $u_1$ or $u_2$, by a tiling argument.

The estimate \eqref{Z:14} is optimal up to an absolute factor, as can be seen by testing it, for any $N > 0$, on
\begin{equation}\label{Z:16}
\left\{
\begin{aligned}
 &\omega,\omega' \in \mathbb S^2, \qquad \theta(\omega,\omega') \sim N^{-1/2},
  \\
  &\widehat u_1(\tau,\xi)
  = \chi_{\tau=\xi\cdot\omega+O(1)}\chi_{\abs{\xi} \sim N}\chi_{\theta(\xi,\omega') \le N^{-1/2}},
  \\
  &\widehat u_2(\tau,\xi)
  = \chi_{\tau=\xi\cdot\omega+O(1)}\chi_{\abs{\xi} \sim N}\chi_{\theta(\xi,\omega) \le N^{-1/2}},
\end{aligned}
\right.
\end{equation}
and with
$$
  I = [-N^{1/2},N^{1/2}], \qquad \alpha \sim N^{-1/2}.
$$
Then \eqref{A:120} holds with $N_1 \sim N_2 \sim N$ and $L_1 \sim L_2 \sim 1$ (by Lemma \ref{D:Lemma4} below), \eqref{Z:10} holds, and the two sides of \eqref{Z:14} are comparable, uniformly in $N$.

The proof of Theorem \ref{Z:Thm1}, given in section \ref{K}, relies in part on the following estimate for the area of intersection between either an ellipsoid or a hyperboloid of revolution, a thickened plane, and a ball centered at one of the foci.

\begin{theorem}\label{Z:Thm2}
Let $a \ge b > 0$, and let $S \subset \R^3$ be the surface obtained by revolving either the ellipse
$$
  \frac{x^2}{a^2} + \frac{y^2}{b^2} = 1
$$
or the hyperbola
$$
  \frac{x^2}{a^2} - \frac{y^2}{b^2} = 1
$$
about the $x$-axis. Let $B \subset \R^3$ be a ball centered at one of the foci of $S$, with radius
\begin{equation}\label{Z:20}
  \frac{b^2}{a} \lesssim R \lesssim a.
\end{equation}
Let $\delta > 0$, and let $P_\delta \subset \R^3$ be any $\delta$-thickened plane. Then
\begin{equation}\label{Z:22}
  \sigma(S \cap B \cap P_\delta) \lesssim R\delta,
\end{equation}
where $\sigma$ denotes surface measure on $S$.
\end{theorem}

The proof is given in section \ref{G:50}. The explanation for the left inequality in \eqref{Z:20} is simply that the minimum distance from $S$ to either of its foci is comparable to $b^2/a$, hence $S \cap B$ is empty unless $R \gtrsim b^2/a$. The right inequality in \eqref{Z:20} is only a restriction when $S$ is a hyperboloid, of course. In that case, \eqref{Z:22} really does fail for $R \gg a$, as can be seen by taking $P_\delta$ to be a thickening of a tangent plane to the asymptotic cone of $S$, with $\delta$ comparable to the minimum distance from the cone to $S \cap B$, namely $\delta \sim ab/R \ll b$; then the area of intersection is comparable to $R\sqrt{R(b/a)\delta} \gg R\sqrt{b\delta} \gg R\delta$.

\subsection{Null form estimates}\label{N}

Here we discuss estimates where the product $u_1u_2$ is replaced by the bilinear operator $\nullform(u_1,u_2)$, defined on the Fourier transform side by inserting the angle $\theta_{12} = \theta(\pm_1\xi_1,\pm_2\xi_2)$ into \eqref{A:130}:\begin{equation}\label{N:10}
  \mathcal F \left( \nullform(u_1,u_2) \right)(X_0)
  = \iint \theta_{12} \,
  \widehat{u_1}(X_1)\widehat{u_2}(X_2)  
  \, \delta(X_0-X_1-X_2) \, dX_1 \, dX_2.
\end{equation}
We call $\nullform$ a \emph{null form}, since it is related to the null forms investigated in \cite{Klainerman:1993}, and subsequently in a number of papers by various authors; see \cite{Foschi:2000} for further references.

In general, the null form improves the estimates. To motivate this heuristically, consider for a moment the generic problem
\begin{equation}\label{N:20}
  (i\partial_t \pm_0 \abs{D}) u_0 = B(u_1,u_2), \qquad u(0) = 0,
\end{equation}
for given $u_1,u_2$, where $B$ is some bilinear operator. This sort of problem would arise as part of an iteration scheme for a nonlinear wave equation, for example. Then $u_1,u_2$ would be previous iterates whose regularity is known, and we want to solve for $u_0$ and find its regularity. Heuristically, this corresponds, in Fourier space, to dividing by the symbol $\hypwt_0$, suggesting that the regularity of $u_0$ depends strongly on the behavior of $\mathcal F B(u_1,u_2)(X_0)$ as $\hypwt_0 \to 0$. Similarly, from the previous level of the iteration, the regularity of $u_1,u_2$ depends the behavior as $\hypwt_1, \hypwt_2 \to 0$. This heuristic suggests that the worst case is when all three hyperbolic weights vanish.

\begin{definition}\label{N:Def1}
Let $(X_0,X_1,X_2)$ be a bilinear interaction (Definition \ref{A:Def}). Given a triple of signs $(\pm_0,\pm_1,\pm_2)$, this bilinear interaction is said to be \emph{null} if all the hyperbolic weights, as defined in \eqref{N:12}, vanish: $\hypwt_0 = \hypwt_1 = \hypwt_2 = 0$.
\end{definition}

In the bilinear null interaction, the $X_j$ all lie on the null cone $K^+\cup K^-$, and since $X_0 = X_1 + X_2$, it is clear that they must be collinear, hence $\theta_{12} = 0$. 

The following lemma generalizes this observation.

\begin{lemma}\label{I:Lemma1} Consider a bilinear interaction $(X_0,X_1,X_2)$, with $\xi_j \neq 0$, $j=0,1,2$. Then for all signs $(\pm_0,\pm_1,\pm_2)$ we have, with notation as in \eqref{N:12} and \eqref{A:170}, 
\begin{equation}\label{I:10}
  \max\left( \abs{\hypwt_0}, \abs{\hypwt_1}, \abs{\hypwt_2} \right)
  \gtrsim
  \min\left(\abs{\xi_1},\abs{\xi_2}\right)\theta_{12}^2.
\end{equation}
Moreover, if
\begin{equation}\label{I:12}
  \abs{\xi_0} \ll \abs{\xi_1} \sim \abs{\xi_2}
  \qquad \text{and} \qquad
  \pm_1 = \pm_2,
\end{equation}
then $\theta_{12} \sim 1$, whereas if \eqref{I:12} does not hold, then
\begin{equation}\label{I:16}
  \max\left( \abs{\hypwt_0}, \abs{\hypwt_1}, \abs{\hypwt_2} \right)
  \gtrsim
  \frac{\abs{\xi_1}\abs{\xi_2}\theta_{12}^2}{\abs{\xi_0}}.
\end{equation}
Furthermore, if the sign $\pm_0$ is chosen so that
\begin{equation}\label{G:8}
  \abs{\hypwt_0} = \bigabs{\abs{\tau_0} - \abs{\xi_0}},
\end{equation}
and if
\begin{equation}\label{G:10}
  \abs{\hypwt_1}, \abs{\hypwt_2} \ll \abs{\hypwt_0},
\end{equation}
then
\begin{equation}\label{G:12}
  \abs{\hypwt_0} \sim
  \left\{
  \begin{alignedat}{2}
  &\min\left(\abs{\xi_1},\abs{\xi_2}\right)\theta_{12}^2& \quad &\text{if $\pm_1=\pm_2$},
  \\
  &\frac{\abs{\xi_1}\abs{\xi_2}\theta_{12}^2}{\abs{\xi_0}}& \quad &\text{if $\pm_1\neq\pm_2$}.
  \end{alignedat}
  \right.
\end{equation}
\end{lemma}

This is proved in section \ref{G}.

As a first example, if we combine \eqref{I:10} with the standard bilinear estimate \eqref{A:114} from Theorem \ref{A:Thm1}, we immediately obtain the following null form estimate:

\begin{corollary}\label{N:Thm1} Assume $u_1,u_2$ satisfy \eqref{A:120}. Then
$$
  \Bignorm{ \Proj_{K^{\pm_0}_{N_0,L_0}}\!\!
  \nullform(u_1,u_2)}
  \lesssim \left( N_0 L_0 L_1 L_2 \right)^{1/2}
  \norm{u_1}\norm{u_2}.
$$
\end{corollary}

In effect, the null symbol $\theta_{12}$ allows us to replace one of the elliptic weights in \eqref{A:114} with a hyperbolic weight, which is good if we are close to a null interaction.

We now present two new null form estimates proved in this paper.

\begin{theorem}\label{N:Thm2} Given $r > 0$ and $\omega \in \mathbb S^2$, let $T_r(\omega) \subset \R^3$ be the tube of radius $r$ around the axis $\R\omega$. Then, assuming $u_1,u_2$ satisfy \eqref{A:120}, 
$$
  \norm{\nullform(\Proj_{\R \times T_r(\omega)} u_1,u_2)}
  \lesssim \left( r^2 L_1 L_2 \right)^{1/2}
  \norm{u_1}\norm{u_2}.
$$
\end{theorem}

The proof is given in section \ref{E}. This estimate is optimal up to an absolute factor, as can be seen by testing it on \eqref{Z:16}, with $r \sim N^{1/2}$.

The key point in the above result is that we are able to exploit concentration of the Fourier supports along null rays; for a standard product such concentrations cannot give any improvement, since in the worst case the thickened cones already intersect along a null ray (approximately, assuming $L_1, L_2$ small relative to $N_1, N_2$).

We shall also prove the following related result, where instead of a tube we have a ball. This situation is of course much better, and we can in fact replace the symbol $\theta_{12}$ by $\sqrt{\theta_{12}}$; the corresponding null form is denoted $\nullformhalf$. The following should be compared with \eqref{X:20}. 

\begin{theorem}\label{N:Thm3} Given $r > 0$, let $B \subset \R^3$ be a ball of radius $r$, with arbitrary center. Then, assuming $u_1,u_2$ satisfy \eqref{A:120},  
$$
  \norm{\nullformhalf(\Proj_{\R \times B} u_1,u_2)}
  \lesssim \left( r^2 L_1 L_2 \right)^{1/2}
  \norm{u_1}\norm{u_2}.
$$
\end{theorem}

Again this is optimal, as can be seen by testing it on the modification of \eqref{Z:16} where we shorten the $\xi$-supports to a length $r \sim N^{1/2}$ in the radial direction.

We remark that since the center of the ball $B$ is arbitrary, the theorem still holds if $\Proj_{\R \times B}$ is placed outside the product, by a tiling argument. (This would not work for Theorem \ref{N:Thm2}, since there we need the tube to pass through the origin.)

\subsection{Concentration/nonconcentration null form estimate}\label{Z}

By analogy with \eqref{X:18} and \eqref{X:20}, the question naturally arises whether we can see an improvement in Theorem \ref{N:Thm2} if we restrict the spatial output frequency $\xi_0$ to a ball $B \subset \R^3$ of radius $\delta$ and arbitrary center. Thus, we consider
\begin{equation}\label{N:100}
  \norm{\Proj_{\R \times B} \nullform(\Proj_{\R \times T_r(\omega)} u_1,u_2)}.
\end{equation}
We obviously get an improvement if $\delta \lesssim r$, since then we can apply Theorem \ref{N:Thm3}, obtaining $\eqref{N:100} \lesssim (\delta^2L_1L_2)^{1/2}\norm{u_1}\norm{u_2}$. So let us assume $B$ has radius $\delta \gg r$. Then Fourier restriction to $B$ will have no effect in directions perpendicular to $\omega$, so we may as well replace \eqref{N:100} by
\begin{equation}\label{N:100:2}
  \norm{\Proj_{\xi_0 \cdot \omega \in I_0} \nullform(\Proj_{\R \times T_r(\omega)} u_1,u_2)},
\end{equation}
where $I_0 \subset \R$ is a compact interval of length $\abs{I_0} = \delta$. Now we test this on \eqref{Z:16}, where we recall that $N_1 \sim N_2 \sim N$. Taking $r \sim N^{1/2}$ and $N^{1/2} \ll \delta \lesssim N$, we observe an improvement by a factor $(\delta/N)^{1/2}$ over the estimate in Theorem \ref{N:Thm2}:
$$
  \eqref{N:100:2} \lesssim \left(r^2L_1L_2\right)^{1/2}
  \left(\frac{\delta}{N_1}\right)^{1/2}\norm{u_1}\norm{u_2}
  \qquad \text{in \eqref{Z:16}}.
$$
On the other hand,\footnote{In general, this holds holds with $\sim$ replaced by $\lesssim$.}
$$
  \left(\frac{\delta}{N_1}\right)^{1/2}
  \norm{u_1}
  \sim \sup_{I_1} \norm{\Proj_{\xi_1 \cdot \omega \in I_1} u_1}
  \qquad \text{in \eqref{Z:16}},
$$
where the supremum is over all translates $I_1$ of $I_0$. This shows that the following result is optimal, up to an absolute factor. Here we assume $r \ll \Nmin^{12}$, since Theorem \ref{Z:Thm1} is the natural result to use if $r \gtrsim \Nmin^{12}$. Moreover, we shall limit attention to interactions which are nearly null, by restricting the symbol in \eqref{N:10} to $\theta_{12} \ll 1$; we denote this modified null form by $\truenullform$.
 
\begin{theorem}\label{N:Thm4}
Let $r > 0$, $\omega \in \mathbb S^2$ and $I_0 \subset \R$ a compact interval. Assume that $u_1,u_2$ satisfy \eqref{A:120}, and that $r \ll \Nmin^{12}$. Then
$$
  \norm{\Proj_{\xi_0 \cdot \omega \in I_0} \truenullform(\Proj_{\R \times T_r(\omega)} u_1,u_2)}
  \lesssim \left( r^2 L_1 L_2 \right)^{1/2} 
  \left( \sup_{I_1} \norm{\Proj_{\xi_1 \cdot \omega \in I_1} u_1} \right)
  \norm{u_2},
$$
where the supremum is over all translates $I_1$ of $I_0$.
\end{theorem}

The proof is given in section \ref{J}.

Note that Theorem \ref{N:Thm4} is better than Theorem \ref{N:Thm2} if the the $\omega$-component of the spatial Fourier support of $u_1$ is not highly concentrated. Thus, a unidirectional concentration of the output frequency $\xi_0$ leads to an improvement if either of the input frequencies $\xi_1,\xi_2$ exhibits nonconcentration for the same direction.

The restriction to interactions with $\theta_{12} \ll 1$ is probably not essential, but we do not pursue this issue here.

We remark that if $u_1$ is itself a product, then Theorem \ref{N:Thm4} may be applied to good effect in combination with Theorem \ref{Z:Thm1}.

\subsection{A nonconcentration low output estimate}\label{L}

In the case $N_0 \ll N_1 \sim N_2$ (low output), \eqref{A:114} says that
$$
  \Bignorm{\Proj_{K^{\pm_0}_{N_0,L_0}}\!\!(u_1u_2)}
  \lesssim \left( N_0 N_1 \Lmin^{012} \Lmed^{012} \right)^{1/2} \norm{u_1} \norm{u_2}.
$$
In general this is optimal, as can be seen by testing it on functions which concentrate along a null ray in Fourier space, but one may hope to do better if the Fourier supports are less concentrated. To detect radial nonconcentration we introduce a modified $L^2$ norm as follows. Let
$$
  \Omega(\gamma) \subset \mathbb S^2 \qquad (0 < \gamma \le \pi )
$$
be a maximal $\gamma$-separated subset of the unit sphere $\mathbb S^2$. Since the cardinality of $\Omega(\gamma)$ is comparable to $1/\gamma^2$, we see that for any $N,r > 0$,
\begin{equation}\label{L:10}
\begin{aligned}
  \norm{u}
  &\sim \left( \sum_{\omega \in \Omega(\frac{r}{N})} \norm{\Proj_{\Delta B_N \cap T_r(\omega)} u}^2 \right)^{1/2}
  \\
  &\lesssim
  \norm{u}_{N,r}
  \equiv
  \frac{N}{r}
  \sup_{\omega \in \mathbb S^2}
  \norm{\Proj_{\Delta B_N \cap T_r(\omega)} u},
\end{aligned}
\end{equation}
and the less radial concentration we have in the spatial Fourier support, the closer the two norms are to being comparable. In the extreme case of spherical symmetry in $\xi$, we have $\norm{u} \sim \norm{u}_{N,r}$. 

We then have the following result.

\begin{theorem}\label{L:Thm1} Assume $N_0 \ll N_1 \sim N_2$, and define $r = (N_0\Lmax^{012})^{1/2}$. Assume as usual that $u_1,u_2 \in L^2(\R^{1+3})$ satisfy \eqref{A:120}. Then
$$
  \Bignorm{\Proj_{K^{\pm_0}_{N_0,L_0}}\!\!(u_1u_2)}
  \lesssim \left( N_0 L_0 L_1 L_2 \right)^{1/2} \norm{u_1}
  \norm{u_2}_{N_2,r}.
$$
In other words,
$$
  \Bignorm{\Proj_{K^{\pm_0}_{N_0,L_0}}\!\!(u_1u_2)}
  \lesssim \left( N_1^2 \Lmin^{012} \Lmed^{012} \right)^{1/2} \norm{u_1}
  \sup_{\omega \in \mathbb S^2} \norm{\Proj_{\R \times T_{r}(\omega)}u_2}.
$$
\end{theorem}

The proof, given in section \ref{D}, relies on two separate angular decompositions based on relationships among the angles between the spatial frequencies $\xi_0,\xi_1,\xi_2$. For the angularly localized pieces we apply the standard bilinear estimates from Theorem \ref{A:Thm1}, and in the summation over the angular sectors we use the following lemma, which is a partial orthogonality result for a family of thickened null hyperplanes corresponding to a set of well-separated directions on the unit sphere.

We introduce the notation
$$
  H_d(\omega) = \left\{ (\tau,\xi) \in \R^3 \colon \Abs{-\tau + \xi \cdot \omega} \le d \right\} \qquad (d > 0, \; \omega \in \mathbb S^2)
$$
for a thickening of the null hyperplane $-\tau + \xi \cdot \omega = 0$.

\begin{lemma}\label{L:Lemma} Suppose $N,d > 0$, $\omega_0 \in \mathbb S^2$ and $0 < \gamma < \gamma' < 1$. Then
\begin{equation}\label{L:40}
  \sum_{\genfrac{}{}{0pt}{1}{\omega \in \Omega(\gamma)}{\theta(\omega,\omega_0) \le \gamma'}} \chi_{H_d(\omega)}(\tau,\xi)
  \lesssim \frac{\gamma'}{\gamma} + \frac{d}{N\gamma^2}
  \qquad \text{for all $(\tau,\xi) \in \R^{1+3}$ with $\abs{\xi} \sim N$.}
\end{equation}
\end{lemma}

This is also proved in section \ref{D}. Observe that the cardinality of the index set in the sum on the left is comparable to the square of the first term on the right.

\section{General setup for bilinear restriction estimates}\label{B}

Here we give a concise summary of the general philosophy behind proving bilinear $L^2$ restriction estimates of the form \eqref{A:10}. A much more wide-ranging discussion of this theme can be found in \cite{Tao:2001}.

The discussion applies to $\R^n$, any $n \ge 1$. By duality, \eqref{A:10} is equivalent to
\begin{equation}\label{B:4}
  \Abs{I} \le C_{A_0,A_1,A_2}
  \norm{\Proj_{-A_0}f_0} \norm{\Proj_{A_1}f_1} \norm{\Proj_{A_2}f_2},
\end{equation}
where
$$
  I = \int_{\R^n} f_0(x) \Proj_{A_0}(\Proj_{A_1}f_1 \cdot \Proj_{A_2}f_2)(x)
  \, dx
  = \int_{\R^n} \Proj_{-A_0}f_0(x) \Proj_{A_1}f_1(x) \Proj_{A_2}f_2(x) \, dx,
$$
and we used Plancherel's theorem to get the last equality. Here $-A_0$ is the reflection of $A_0$ about the origin. Once the estimate is written in this way, it becomes clear that \eqref{A:10} enjoys a permutation invariance, conveniently summarized in the rule
\begin{equation}\label{B:10}
  C_{A_0,A_1,A_2} = C_{-A_2,-A_0,A_1} = C_{-A_1,-A_0,A_2}.
\end{equation}

On the other hand, $I$ can also be written in the Fourier variables as
$$
  I = \iint \chi_{\xi_1 \in A_1 \cap (\xi-A_2)} \chi_{\xi \in A_0} \widehat{\Proj_{A_1}f_1}(\xi_1) 
  \widehat{\Proj_{A_2}f_2}(\xi-\xi_1)
  \widehat{\Proj_{-A_0}f_0}(-\xi) \, d\xi_1 \, d\xi.
$$
Then applying the Cauchy-Schwarz inequality twice, first with respect to $\xi_1$ and then with respect to $\xi$, we get \eqref{B:4} with $C_{A_0,A_1,A_2} = \sup_{\xi \in A_0} \abs{A_1 \cap (\xi-A_2)}^{1/2}$. Here $\abs{A}$ denotes the $n$-dimensional volume of a set $A \subset \R^n$.

Using also the permutation rule \eqref{B:10}, we then conclude:

\begin{lemma}\label{B:Lemma1} \eqref{B:4} holds with $(C_{A_0,A_1,A_2})^2$ equal to an absolute constant times
$$
  \min\left(\sup_{\xi \in A_0} \Abs{A_1 \cap (\xi-A_2)},\sup_{\xi \in A_2} \Abs{A_0 \cap (\xi+A_1)},\sup_{\xi \in A_1} \Abs{A_0 \cap (\xi+A_2)}\right),
$$
provided that this quantity is finite.
\end{lemma}

In fact, under certain hypotheses one can take the intersection of translates of all three sets at once, as we now discuss.

We say $A \subset \R^n$ is an \emph{approximate tiling set} if, for some lattice $E \subset \R^n$, $\left\{ \xi + A \right\}_{\xi \in E}$ is a cover of $\R^n$ with $O(1)$ overlap. If the cover is almost disjoint, so that there is essentially no overlap, we just say that $A$ is a \emph{tiling set}.

Here \emph{$O(1)$ overlap} means that there exists a number $M \in \N$ such that for any $\eta \in E$, the number of $\xi \in E$ such that $\xi+A$ and $\eta+A$ intersect is at most $M$. But since $E$ is a lattice, this is equivalent to saying that $\left\{ \xi \in E : (\xi + A) \cap A \neq \emptyset \right\}$ has cardinality at most $M$.

Further, defining $A^* = A + A$, which we call the \emph{doubling} of $A$, we say that $A$ has the \emph{doubling property} if the cover $\left\{ \xi + A^* \right\}_{\xi \in E}$ also has $O(1)$ overlap.

A more general version of the following lemma, but with a less direct proof, can be found in \cite{Tao:2001}. In this lemma and its proof, implicit constants depend on the size of the overlap of the doubling cover (the number $M$ appearing in the proof).

\begin{lemma}\label{B:Lemma2} Suppose $A_0$ is an approximate tiling set with the doubling property. Then \eqref{B:4} holds with
$$
  C_{A_0,A_1,A_2}
  \sim
  \left(\sup_{\xi \in A_0, \; \xi' \in E} \bigabs{A_1 \cap (\xi-A_2) \cap (\xi'+A_0)}\right)^{1/2},
$$
provided that this quantity is finite.
\end{lemma}

\begin{proof}
Represent the lattice $E$ explicitly as
$$
  E = \left\{ k_1v_1+\dots+k_mv_m : k_1,\dots,k_m \in \Z \right\},
$$
where $v_1,\dots,v_m \in \R^n$ are linearly independent. Write
$$
  A_0(k) = k_1v_1+\dots+k_mv_m+A_0,
  \qquad
  \text{for $k = (k_1,\dots,k_m) \in \Z^m$}.
$$
The doubling property implies that, for some $M \in \N$,
\begin{equation}\label{B:20}
  A_0(k) \cap ( A_0 + A_0(l) ) = \emptyset \qquad \text{for all $k,l \in \Z^m$ with $\norm{k-l}_\infty > M$.} 
\end{equation}
Without loss of generality, assume $\widehat{f_1}, \widehat{f_2} \ge 0$. Since $\{ A_0(k) \}_{k \in \Z^m}$ is a cover of $\R^n$,
\begin{equation}\label{B:22}
  \norm{\Proj_{A_0}\left( \Proj_{A_1}f_1 \cdot \Proj_{A_2}f_2 \right)}
  \le \sum_{k,l \in \Z^m} \norm{\Proj_{A_0}
  \left( \Proj_{A_1 \cap A_0(k)}f_1
  \cdot \Proj_{A_2 \cap (-A_0(l))}f_2 \right)}.
\end{equation}
But the summand vanishes unless there exist $\xi_1 \in A_0(k)$ and $\xi_2 \in -A_0(l)$ such that $\xi_1+\xi_2 \in A_0$, implying $A_0(k) \cap ( A_0 + A_0(l) ) \neq \emptyset$. In view of \eqref{B:20} we can therefore restrict the sum in \eqref{B:22} to $\norm{k-l}_\infty \le M$.

Now set
$$
  a_k = \norm{\Proj_{A_1 \cap A_0(k)}f_1},
  \qquad
  b_l = \norm{\Proj_{A_2 \cap (-A_0(l))}f_2}.
$$
By the $O(1)$ overlap of the cover $\{ A_0(k) \}_{k \in \Z^m}$,
\begin{equation}\label{B:24}
  \left( \sum_{k \in \Z^m} a_k^2 \right)^{1/2}
  \sim \norm{\Proj_{A_1}f_1},
  \qquad
  \left( \sum_{l \in \Z^m} b_{l}^2 \right)^{1/2}
  \sim
  \norm{\Proj_{A_2}f_2}.
\end{equation}
Using  Lemma \ref{B:Lemma1}, we then obtain
\begin{align*}
  \text{l.h.s.}\eqref{B:22}
  &\le \sum_{\genfrac{}{}{0pt}{1}{k,l \in \Z^m}{\norm{k-l}_\infty \le M}}
  \Biggl(\sup_{\xi \in A_0} \bigabs{A_1\cap A_0(k) \cap (\xi-A_2)}\Biggr)^{1/2}
  a_k b_l
  \\
  &\le
  \Biggl(\sup_{\xi \in A_0, \; \xi' \in E} \bigabs{A_1 \cap (\xi-A_2) \cap (\xi'+A_0)}\Biggr)^{1/2}
  \sum_{\genfrac{}{}{0pt}{1}{k,l \in \Z^m}{\norm{k-l}_\infty \le M}} a_k b_l,
\end{align*}
and the last sum can be rewritten as
$$
  \sum_{\genfrac{}{}{0pt}{1}{k,l' \in \Z^m}{\norm{l'}_\infty \le M}} a_k b_{k+l'}
  \le\sum_{\genfrac{}{}{0pt}{1}{l' \in \Z^m}{\norm{l'}_\infty \le M}}
  \left( \sum_{k \in \Z^m} a_k^2 \right)^{1/2} 
  \left( \sum_{k \in \Z^m} b_{k+l'}^2 \right)^{1/2}
  \lesssim
  \norm{\Proj_{A_1}f_1}\norm{\Proj_{A_2}f_2},
$$
where we used the Cauchy-Schwarz inequality and \eqref{B:24}. This proves the lemma.
\end{proof}

\section{A unified approach to the main examples}\label{C}

Here we give short, unified proofs of the examples \eqref{A:20}--\eqref{A:24} from the introduction, based just on the elementary estimate from Lemma \ref{C:Lemma1}.

\subsection{Stein-Tomas restriction theorem}\label{C:10}

As noted, this reduces to the bilinear restriction estimate \eqref{A:40} for the thickened unit sphere $S_\varepsilon=S_\varepsilon(1)$, so by Lemma \ref{B:Lemma1} it suffices to prove
\begin{equation}\label{C:20}
  \Abs{S_\varepsilon \cap (\xi_0 + S_\varepsilon)} \lesssim \varepsilon^2
\end{equation}
for arbitrary $\xi_0 \in \R^3$. This clearly fails when $\xi_0$ is close to zero, since then the spheres coalesce, hence the best possible volume estimate is $O(\varepsilon)$, not $O(\varepsilon^2)$. We can avoid this dangerous concentric interaction by a decomposition in the linear estimate \eqref{A:36}: Without loss of generality replace $S_\varepsilon$ by $S_\varepsilon \cap A$ in \eqref{A:20}, where $A$ is the first octant of $\R^3$, and make the same change in \eqref{C:20}. The point of this is that if $\xi_1, \xi_2 \in A$, then $\xi_0 = \xi_1 + \xi_2$ satisfies $\abs{\xi_0} \sim 1$, and the latter condition then defines the set $A_0$ in the setup of Lemma \ref{B:Lemma1}. Thus, it is enough to prove \eqref{C:10} when $\abs{\xi_0} \sim 1$, but this follows from Lemma \ref{C:Lemma1}.

\subsection{Strichartz's $L^4$ estimate for the wave equation}\label{C:30}

As noted, this reduces to the restriction estimate \eqref{A:58} for a thickened, truncated cone. For definiteness, and without loss of generality, we choose $\pm=+$. The equivalent bilinear $L^2$ estimate then reduces, by Lemma \ref{B:Lemma1}, to proving that, for any $(\tau_0,\xi_0) \in \R^{1+3}$, the set
$$
  E = K^+_{N,L} \cap \left( (\tau_0,\xi_0) - K^+_{N,L} \right)
$$
verifies the volume bound
\begin{equation}\label{C:40}
  \Abs{E} \lesssim N^2 L^2.
\end{equation}
The slices $\tau=\text{const}$ are denoted
\begin{equation}\label{C:42}
  E(\tau) = \left\{ \xi \in \R^3 \colon (\tau,\xi) \in E \right\},
\end{equation}
and we define 
\begin{equation}\label{C:44}
  J = \left\{ \tau \in \R \colon E(\tau) \neq \emptyset \right\}.
\end{equation}
Then by Fubini's theorem,
\begin{equation}\label{C:46}
  \Abs{E} \le \Abs{J} \cdot \sup_{\tau \in J} \Abs{E(\tau)}.
\end{equation}
The advantage of slicing by $\tau=\text{const}$ is that such a slice of thickened null cone is nothing else than a thickened sphere, providing an immediate connection with our proof of the Stein-Tomas estimate.

Now observe that
\begin{multline}\label{C:50}
  E \subset \Big\{ (\tau,\xi) \in \R^{1+3} \colon \frac{N}{2} \le \abs{\xi} \le N, \;
  \frac{N}{2} \le \abs{\xi_0-\xi} \le N,
  \\
  \tau=\abs{\xi}+O(L),
  \;
  \tau_0-\tau=\abs{\xi_0-\xi}+O(L) \Big\},
\end{multline}
where $\tau=\abs{\xi}+O(L)$ stands for $\bigabs{\tau-\abs{\xi}} \le L$. Assume $L \ll N$. Then by \eqref{C:50},
\begin{equation}\label{C:52}
\left\{
  \begin{alignedat}{2}
  E(\tau) &\subset S_L(\tau) \cap (\xi_0 + S_L(\tau_0-\tau))& \qquad &\text{if $\tau \sim N$},
  \\
  E(\tau) &= \emptyset& \qquad &\text{otherwise}.
  \end{alignedat}
\right.
\end{equation}
Therefore, $\abs{J} \lesssim N$, and by Lemma \ref{C:Lemma1},
\begin{equation}\label{C:54}
  \Abs{E(\tau)} \lesssim \frac{N^2L^2}{\abs{\xi_0}},
\end{equation}
hence \eqref{C:40} follows from \eqref{C:46} provided that $\abs{\xi_0} \sim N$. But this we can ensure by the same trick as we used for the Stein-Tomas theorem: We can replace $K^\pm_{N,L}$ by $K^\pm_{N,L} \cap (\R \times A)$ in the original estimate \eqref{A:58}, $A$ being the first octant of $\R^3_\xi$. This concludes the proof for $L \ll N$. If $L \gtrsim N$, on the other hand, then we use the fact that $\Abs{E} \lesssim N^3 L$, as is obvious from \eqref{C:50}.

This concludes the proof of Strichartz's estimate. Before we move on, however, let us note that the above proof easily gives also the following modified estimate, which is a special case of a theorem proved in \cite{Klainerman:1999}.

\begin{theorem}\label{C:Thm} \emph{(\cite{Klainerman:1999}.)} Let $N,L > 0$, and let $B \subset \R^3$ be a ball of radius $r \ll N$, with arbitrary center. Then
$$
  \norm{\Proj_{K^\pm_{N,L} \cap (\R \times B)} u}_{L^4(\R^{1+3})}
  \lesssim \left( (rN)^{1/2} L \right)^{1/2}
  \norm{\Proj_{K^\pm_{N,L} \cap (\R \times B)} u}_{L^2(\R^{1+3})}.
$$
\end{theorem}

To prove this, we repeat the above argument. We need $\Abs{E} \lesssim rNL^2$, where now
$$
  E \subset \Big\{ (\tau,\xi) \colon \xi \in B, \;
  \xi_0-\xi \in B,
  \;
  \tau=\abs{\xi}+O(L),
  \;
  \tau_0-\tau=\abs{\xi_0-\xi}+O(L) \Big\}.
$$
Thus, $\xi_0 \in B + B$, hence $\abs{\xi_0} \sim N$, so the right side of \eqref{C:54} is comparable to $NL^2$. Clearly, \eqref{C:52} holds with $\tau \sim N$ replaced by $\tau=\abs{\xi^*} + O(\max(r,L))$, where $\xi_*$ is the center of $B$. So if $L \lesssim r$, then $\abs{J} \lesssim r$, and we conclude from \eqref{C:54} and \eqref{C:46} that $\Abs{E} \lesssim rNL^2$, as desired. On the other hand, $\Abs{E} \lesssim r^3 L$, covering the case $L \gtrsim r$.

\subsection{Bilinear generalization of Strichartz's estimate}\label{C:80}

As noted, this reduces to \eqref{A:76}. We first prove the version where the balls $B_{N_j}$ are replaced by the annuli $\Delta B_{N_j}$, then in subsection \ref{C:100} we show how to generalize to balls.

Without loss of generality, we assume $N_1 \le N_2$. Then by \eqref{A:140} we split into the cases $N_1 \lesssim N_0 \sim N_2$ and $N_0 \ll N_1 \sim N_2$, hence we need to prove
\begin{alignat}{2}
  \label{C:92}
  \norm{\Proj_{\R \times \Delta B_{N_0}}(u_1 u_2)}
  &\lesssim \left(N_0N_1L_1L_2\right)^{1/2}
  \norm{u_1}\norm{u_2}&
  \qquad &\text{if $N_0 \ll N_1 \sim N_2$},
  \\
  \label{C:90}
  \norm{\Proj_{\R \times \Delta B_{N_0}}(u_1 u_2)}
  &\lesssim \left(N_1^2L_1L_2\right)^{1/2}
  \norm{u_1}\norm{u_2}&
  \qquad &\text{if $N_1 \lesssim N_0 \sim N_2$},
\end{alignat}
for $u_1,u_2$ satisfying \eqref{A:120}.

\subsubsection{Low output case: $N_0 \ll N_1 \sim N_2$}

By tiling, as in the proof of Lemma \ref{B:Lemma2}, we can reduce \eqref{C:92} to proving, for arbitrary translates $B,B'$ of $B_{N_0}$,
$$
  \norm{\Proj_{\R \times B}u_1 \cdot \Proj_{\R \times B'}u_2}
  \lesssim \left(N_0N_1L_1L_2\right)^{1/2}
  \norm{\Proj_{\R \times B}u_1}\norm{\Proj_{\R \times B'}u_2},
$$
but by H\"older's inequality this reduces to Theorem \ref{C:Thm}, proved above.

\subsubsection{High output case: $N_1 \lesssim N_0 \sim N_2$}

It suffices to prove \eqref{C:90} with $\pm_1=+$. We now argue as in subsection \ref{C:30}, but with
$$
  E = K^+_{N_1,L_2} \cap \left( (\tau_0,\xi_0) - K^\pm_{N_2,L_2} \right)
$$
for some $(\tau_0,\xi_0)$ with $\abs{\xi_0} \sim N_0$ (due to the restriction to $\Delta B_{N_0}$), and we need
\begin{equation}\label{C:92:2}
  \Abs{E} \lesssim N_1^2 L_1L_2.
\end{equation}
Note that
\begin{multline}\label{C:94}
  E \subset \Big\{ (\tau,\xi) \in \R^{1+3} \colon \frac{N_1}{2} \le \abs{\xi} \le N_1, \;
  \frac{N_2}{2} \le \abs{\xi_0-\xi} \le N_2,
  \\
  \tau=\abs{\xi}+O(L_1),
  \;
  \tau_0-\tau=\pm\abs{\xi_0-\xi}+O(L_2) \Big\}.
\end{multline}
Assuming for the moment $L_1,L_2 \ll N_1$, we see from \eqref{C:94} that
$$
\left\{
  \begin{alignedat}{2}
  E(\tau) &\subset S_{L_1}(\tau) \cap (\xi_0 + S_{L_2}(\tau_0-\tau))& \qquad &\text{if $\tau \sim N_1$},
  \\
  E(\tau) &= \emptyset& \qquad &\text{otherwise}.
  \end{alignedat}
\right.
$$
Therefore, $\abs{J} \lesssim N_1$, and by Lemma \ref{C:Lemma1}, recalling also that $\abs{\xi_0} \sim N_0 \sim N_2$,
\begin{equation}\label{C:96}
  \Abs{E(\tau)} \lesssim \frac{N_1N_2L_1L_2}{\abs{\xi_0}} \sim N_1 L_1 L_2.
\end{equation}
In view of \eqref{C:46}, this proves \eqref{C:92:2} when $L_1,L_2 \ll N_1$. If $\Lmax^{12} \gtrsim N_1$, on the other hand, then we can use $\Abs{E} \lesssim N_1^3 \Lmin^{12}$, which is obvious from \eqref{C:94}.

\subsubsection{From annuli to balls}\label{C:100}

Without loss of generality, assume the $N_j$ are dyadic, i.e., they are of the form $2^m$ with $m \in \Z$. Write $B_{N_j}$ as an almost disjoint union
$$
  B_{N_j} = \bigcup_{0 < N_j' \le N_j} \Delta B_{N_j'},
$$
for dyadic $N_j'$. Using also $L^2$ duality to rewrite \eqref{A:76} as a trilinear integral estimate, we then see that \eqref{A:76} reduces to proving
\begin{equation}\label{C:110}
  \sum_{N_0',N_1',N_2'}
  \Abs{\iint \overline{u_0^{N_0'}} u_1^{N_1'} u_2^{N_2'} \, dt \, dx}
  \lesssim \left( \Nmin^{012} \Nmin^{12} L_1 L_2 \right)^{1/2}
  \norm{u_0} \norm{u_1} \norm{u_2}.
\end{equation}
Here the sum is restricted to dyadic $N_j' \in (0,N_j]$, for $j=0,1,2$, we assume $$
  \supp \widehat{u_0} \subset \R \times B_{N_0}
  \qquad
  \supp \widehat{u_j} \subset \left( \R \times B_{N_j} \right) \cap K^{\pm_j}_{L_j} \quad \text{for $j=1,2$},
$$  
and we write $u_j^{N_j'} = \Proj_{\R \times \Delta B_{N_j'}} u_j$ for $j=0,1,2$. As we just proved,
$$
  \Abs{\iint \overline{u_0^{N_0'}} u_1^{N_1'} u_2^{N_2'} \, dt \, dx}
  \lesssim \left( \Nmin'^{\,012} \Nmin'^{\,12} L_1 L_2 \right)^{1/2}
  \bignorm{u_0^{N_0'}} \bignorm{u_1^{N_1'}} \bignorm{u_2^{N_2'}},
$$
so to get \eqref{C:110} it suffices to show
\begin{equation}\label{C:120}
  \sum_{N_0',N_1',N_2'}
  \left( \Nmin'^{\,012} \right)^{1/2}
  \bignorm{u_0^{N_0'}} \bignorm{u_1^{N_1'}} \bignorm{u_2^{N_2'}}
  \lesssim
  \left( \Nmin^{012} \right)^{1/2}
  \bignorm{u_0} \bignorm{u_1} \bignorm{u_2}.
\end{equation}
By symmetry, assume $\Nmin'^{\,012} = N_0'$. Then $N_1' \sim N_2'$, by \eqref{A:140}. Now sum using $\sum_{N_0' \le N_0} (N_0')^{1/2} \sim N_0^{1/2}$
and $\sum_{N_1' \sim N_2'} \bignorm{u_1^{N_1'}} \bignorm{u_2^{N_2'}} \lesssim \bignorm{u_1} \bignorm{u_2}$, where the latter holds by the Cauchy-Schwarz inequality. This proves \eqref{C:120}.

\begin{remark}\label{C:Rem}
In the above proofs, we divided into cases depending on whether the $L$'s are small or not, but by a general argument we can assume the $L$'s arbitrarily small. For example, say we know \eqref{C:90} for $L_1 \lesssim \delta$, some $\delta > 0$. Then to prove \eqref{C:90} for large $L_1$, we cut $\tau_1=\pm_1\abs{\xi_1} + O(L_1)$ into thinner cones $\tau_1=\pm_1\abs{\xi_1} + c + O(\delta)$. For each piece, there is a translation by $c$ in the $\tau_1$-direction, but in physical space this corresponds to multiplying $u_1$ by $e^{itc}$, which does not affect the norms in \eqref{C:90}. Since there are $O(L_1/\delta)$ pieces, summing the individual estimates and using the Cauchy-Schwarz inequality gives us the factor $L_1^{1/2}$ in the right side of \eqref{C:90}.
\end{remark}

\section{Proof of the nonconcentration low output estimate}\label{D}

Here we first prove Theorem \ref{L:Thm1} using Lemma \ref{L:Lemma}, and then we prove the lemma. In preparation for this, we first introduce some notation for angular decompositions.

For $\gamma > 0$ and $\omega \in \mathbb S^2$ we define the conical sector
$$
  \Gamma_\gamma(\omega) = \left\{ \xi \in \R^3 : \theta(\xi,\omega) \le \gamma \right\}.
$$
Denote by $\Omega(\gamma)$ a maximal $\gamma$-separated subset of $\mathbb S^2$. Then
\begin{equation}\label{D:206}
  1 \le \sum_{\omega \in \Omega(\gamma)} \chi_{\Gamma_\gamma(\omega)}(\xi)
  \le 5^2
  \qquad (\forall \xi \neq 0),
\end{equation}
where the left inequality holds by the maximality of $\Omega(\gamma)$, and the right inequality by the $\gamma$-separation, since the latter implies (we omit the proof):
\begin{lemma}\label{D:Lemma1} For $k \in \N$ and $\omega \in \mathbb S^2$,
$\#\left\{ \omega' \in \Omega(\gamma) : \theta(\omega',\omega) \le k\gamma \right\}\le (2k+1)^2$.
\end{lemma}

The following will be used for angular decomposition in bilinear estimates.

\begin{lemma}\label{D:Lemma2} Let $\gamma^* \in (0,1]$ and $m \ge 3$. Define $M =  2 ( 1 + \frac{m+2}{\gamma^*})$. Then
\begin{equation}\label{D:210}
  1 \le \sum_{\genfrac{}{}{0pt}{1}{0 < \gamma \le \gamma^*}{\text{$\gamma$ dyadic}}} 
  \sum_{\genfrac{}{}{0pt}{1}{\omega_1,\omega_2 \in \Omega(\gamma)}{m\gamma \le \theta(\omega_1,\omega_2) \le M\gamma}}
  \chi_{\Gamma_\gamma(\omega_1)}(\xi_1) \chi_{\Gamma_\gamma(\omega_2)}(\xi_2)
  \; \lesssim \; C(M),
\end{equation}
for all $\xi_1,\xi_2 \in \R^3 \setminus \{0\}$ with $\theta(\xi_1,\xi_2) > 0$.
\end{lemma}

We omit the straightforward proof. The condition $\theta(\omega_1,\omega_2) \ge m\gamma$ ensures that the sectors in \eqref{D:210} are well-separated, since $m \ge 3$. If separation is not needed, the following variation, whose proof we also omit, may be used:

\begin{lemma}\label{D:Lemma3} For any $0 < \gamma < 1$ and $k \in \N$,
$$
  \chi_{\theta(\xi_1,\xi_2) \le k\gamma}(\xi_1,\xi_2) \lesssim \sum_{\genfrac{}{}{0pt}{1}{\omega_1,\omega_2 \in \Omega(\gamma)}{\theta(\omega_1,\omega_2) \le (k+2)\gamma}}
  \chi_{\Gamma_\gamma(\omega_1)}(\xi_1) \chi_{\Gamma_\gamma(\omega_2)}(\xi_2),
$$
for all $\xi_1,\xi_2 \in \R^3 \setminus \{0\}$.
\end{lemma}

Recall the notation \eqref{A:120:2}, which we can also restate as
\begin{equation}\label{D:36}
  \widehat{u_j^{\gamma,\omega}}(X_j) = \chi_{\Gamma_\gamma(\omega)}(\pm_j\xi_j)
  \widehat u_j(X_j) \qquad (0 < \gamma < 1, \; \omega \in \mathbb S^2).
\end{equation}
Here $j=1,2$, but later we also use this for $j=0$, if $\pm_0$ is given. Then by \eqref{D:206},
\begin{equation}\label{D:70}
  \bignorm{u_j}
  \sim
  \left( \sum_{\omega \in \Omega(\gamma)}
  \norm{u_j^{\gamma,\omega}}^2 \right)^{1/2}.
\end{equation}
Summing out the $\omega$'s in a bilinear estimate is never a problem. In fact,
\begin{equation}\label{D:38}
\begin{aligned}
  \sum_{\genfrac{}{}{0pt}{1}{\omega_1,\omega_2 \in \Omega(\gamma)}
  {\theta(\omega_1,\omega_2) \lesssim \gamma}}
  \norm{u_1^{\gamma,\omega_1}}\norm{u_2^{\gamma,\omega_2}}
  &\le \left( \sum_{\omega_1,\omega_2}
  \norm{u_1^{\gamma,\omega_1}}^2 \right)^{1/2}
  \left( \sum_{\omega_1,\omega_2}
  \norm{u_2^{\gamma,\omega_2}}^2 \right)^{1/2}
  \\
  &\lesssim \norm{u_1}\norm{u_2},
\end{aligned}
\end{equation}
Here we first applied the Cauchy-Schwarz inequality, then in the second step we used Lemma \ref{D:Lemma1}, and to get the last inequality we used \eqref{D:70}.

We now prove Theorem \ref{L:Thm1}, assuming throughout $L_0 \ll N_0 \ll N_1 \sim N_2$. We split into the cases $\pm_1\neq\pm_2$ and $\pm_1=\pm_2$. Define $\theta_{12}$ as in \eqref{A:170}.

\subsection{The case $\pm_1\neq\pm_2$} Then $\theta_{12}=\theta(\xi_1,-\xi_2)$, and since $\xi_0=\xi_1+\xi_2$ with $\abs{\xi_0} \ll \abs{\xi_1} \sim \abs{\xi_2}$, we conclude that $\theta_{12} \ll 1$, hence $\sin \theta_{12} \sim \theta_{12}$. By Lemma \ref{I:Lemma1},
\begin{equation}\label{D:8}
  \theta_{12} \lesssim \gamma_{12} \equiv \left( \frac{N_0\Lmax^{012}}{N_1^2} \right)^{1/2}.
\end{equation}
Next, define
$$
  \theta_{01} = \theta(\pm_0'\xi_0,\xi_1),
$$
where the sign $\pm_0'$ is chosen so that $\theta_{01} \in [0,\pi/2]$, hence $\sin \theta_{01} \sim \theta_{01}$. To be precise, we split the region of integration into two parts, one for each choice of sign. By the sine rule (see Figure \ref{fig:1}) we then get
\begin{equation}\label{D:12}
  N_0 \theta_{01} \sim N_2 \theta_{12} \sim N_1 \theta_{12},
\end{equation}
and combining this with \eqref{D:8} we conclude that
\begin{equation}\label{D:20}
  \theta_{01} \sim \frac{N_1}{N_0} \theta_{12}
  \lesssim \gamma_{01} \equiv \frac{N_1}{N_0}\gamma_{12} \sim
  \left(\frac{\Lmax^{012}}{N_0}\right)^{1/2}.
\end{equation}

\begin{figure}
   \centering
   \includegraphics{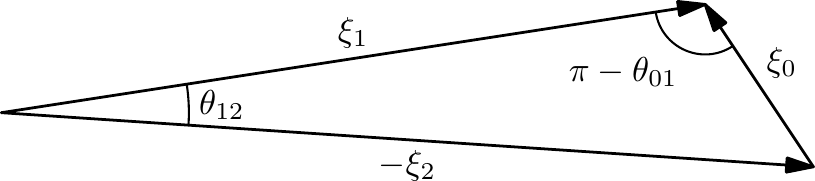}
   \caption{Low output with $\pm_1\neq\pm_2$. In the situation drawn here, $\pm_0'$ is a $-$ sign, to ensure $\theta_{01} = \theta(-\xi_0,\xi_1) \in [0,\pi/2]$.}
   \label{fig:1}
\end{figure}

Rewrite the estimate in Theorem \ref{L:Thm1} in the equivalent form, by duality,
\begin{equation}\label{D:30}
  \Abs{\iint \overline{u_0}\, u_1 u_2 \, dt \, dx}
  \lesssim \left( N_1^2 \Lmin^{012} \Lmed^{012} \right)^{1/2} \norm{u_0} \norm{u_1}
  \sup_{\omega \in \mathbb S^2} \norm{\Proj_{\R \times T_{r}(\omega)}u_2},
\end{equation}
where $u_0,u_1,u_2 \in L^2(\R^{1+3})$ and $u_1,u_2$ satisfy \eqref{A:120}. Without loss of generality, we can assume $\widehat{u_j} \ge 0$ for $j=0,1,2$, hence we can remove the absolute value above.

Now we make an angular decomposition with respect to the maximal dyadic size $\gamma_{01}$ of the angle $\theta_{01}$, given by \eqref{D:20}. By Lemma \ref{D:Lemma3},
\begin{equation}\label{D:34}
  \iint \overline{u_0}\, u_1 u_2 \, dt \, dx
  \lesssim
  \sum_{\genfrac{}{}{0pt}{1}{\omega_0,\omega_1 \in \Omega(\gamma_{01})}
  {\theta(\omega_0,\omega_1) \lesssim \gamma_{01}}}
  \iint \overline{u_0^{\gamma_{01},\omega_0}} u_1^{\gamma_{01},\omega_1} u_2 \, dt \, dx,
\end{equation}
with notation as in \eqref{D:36}, where for $u_0$ we use the sign $\pm_0'$, not $\pm_0$.

Next, we make an additional decomposition with respect to the maximal dyadic size $\gamma_{12}$ of the angle $\theta_{12}$, given by \eqref{D:8}. Thus, applying Lemma \ref{D:Lemma3} one more time,
\begin{multline}\label{D:40}
  \iint \overline{u_0^{\gamma_{01},\omega_0}} u_1^{\gamma_{01},\omega_1} u_2 \, dt \, dx
  \\
  \lesssim
  \sum_{\genfrac{}{}{0pt}{1}{\omega_1',\omega_2' \in \Omega(\gamma_{12})}
  {\theta(\omega_1',\omega_2') \lesssim \gamma_{12}}}
  \iint \overline{u_0^{\gamma_{01},\omega_0}} (u_1^{\gamma_{01},\omega_1})^{\gamma_{12},\omega_1'} u_2^{\gamma_{12},\omega_2'} \, dt \, dx,
\end{multline}
where we use again the notation from \eqref{D:36}. In particular,
$$
  \mathcal F (u_1^{\gamma_{01},\omega_1})^{\gamma_{12},\omega_1'}(X_1)
  = \chi_{\Gamma_{\gamma_{01}}(\omega_1)}(\pm_1\xi_1)
  \chi_{\Gamma_{\gamma_{12}}(\omega_1')}(\pm_1\xi_1)
  \widehat u_1(X_1),
$$
so once $\omega_1$ has been chosen, $\omega_1'$ is constrained by $\theta(\omega_1',\omega_1) \lesssim \gamma_{01}$.

We need the following lemma. Here we use the notation
$$
  H_d(\omega) = \left\{ (\tau,\xi) \in \R^3 \colon \Abs{-\tau + \xi \cdot \omega} \lesssim d \right\} \qquad (d > 0, \; \omega \in \mathbb S^2)
$$
for a thickened null hyperplane (we include an implicit absolute constant to clean up the notation), and $K^{\pm}_{N,L,\gamma,\omega}$ is defined as in \eqref{D:42}.

\begin{lemma}\label{D:Lemma4} For $N,L > 0$, $\omega \in \mathbb S^2$ and $0 < \gamma < 1$,
$$
  K^{\pm}_{N,L,\gamma,\omega} \subset H_{\max(L,N\gamma^2)}(\omega).
$$
\end{lemma}

\begin{proof}
Let $(\tau,\xi) \in K^{\pm}_{N,L,\gamma,\omega}$. Then $-\tau+\xi \cdot \omega$ equals
$$
  \left(-\tau\pm\abs{\xi}\right)
  - \left( \pm \abs{\xi} - \xi \cdot \omega \right)
  = O(L)
  - \frac{\abs{\xi}^2\bigl(1-\cos^2\theta(\pm\xi,\omega)\bigr)}{\pm\left(\abs{\xi}\pm\xi \cdot \omega\right)}
  = O(L) + O(N\gamma^2),
$$
where we used the fact that $\theta(\pm\xi,\omega) \le \gamma < 1$, hence $\pm\xi \cdot \omega \ge 0$.
\end{proof}

If $X_1,X_2$ belong to the Fourier supports of $(u_1^{\gamma_{01},\omega_1})^{\gamma_{12},\omega_1'}, u_2^{\gamma_{12},\omega_2'}$, respectively, then (since $\theta(\omega_1',\omega_2') \lesssim \gamma_{12}$)
$$
  X_j \in \left\{ (\tau,\xi) \in K^{\pm_j}_{N_j,L_j} \colon \theta(\pm_j\xi_j,\omega_1') \lesssim \gamma_{12} \right\} \qquad (j=1,2),
$$
so by Lemma \ref{D:Lemma4} we conclude that $X_j \in H_{\max(L_j,N_j\gamma_{12}^2)}(\omega_1')$ for $j=1,2$, hence
\begin{equation}\label{D:50}
  X_0 = X_1 + X_2
  \in H_d(\omega_1'),
  \qquad \text{where}
  \qquad
  d = \max(\Lmax^{12},N_1\gamma_{12}^2).
\end{equation}
Therefore, we can replace $u_0^{\gamma_{01},\omega_0}$ in \eqref{D:40} by $\Proj_{H_d(\omega_1')} u_0^{\gamma_{01},\omega_0}$, so combining \eqref{D:34} and \eqref{D:40}, and applying \eqref{A:110} or \eqref{A:112} from Theorem \ref{A:Thm1} to each term in \eqref{D:40},
\begin{multline}\label{D:60}
  \iint \overline{u_0}\, u_1 u_2 \, dt \, dx
  \lesssim
  \sum_{\omega_0,\omega_1} \sum_{\omega_1',\omega_2'}
  \left[ \min\left( N_0N_1L_1L_2, N_0^2 L_0 \Lmin^{12} \right) \right]^{1/2}
  \\
  \times
  \Bignorm{\Proj_{H_d(\omega_1')} u_0^{\gamma_{01},\omega_0}}
  \norm{(u_1^{\gamma_{01},\omega_1})^{\gamma_{12},\omega_1'}}
  \norm{u_2^{\gamma_{12},\omega_2'}},
\end{multline}
where the sum is over $\omega_0,\omega_1 \in \Omega(\gamma_{01})$ with $\theta(\omega_0,\omega_1) \lesssim \gamma_{01}$, and $\omega_1',\omega_2' \in \Omega(\gamma_{12})$ with $\theta(\omega_1',\omega_2) \lesssim \gamma_{12}$ and $\theta(\omega_1',\omega_1) \lesssim \gamma_{01}$.

Note that once $\omega_1'$ has been chosen, then the choice of $\omega_2'$ is limited to a set of cardinality $O(1)$, in view of Lemma \ref{D:Lemma1}, and similarly for the pair $\omega_0,\omega_1$. This fact will be used without further mention. In essence, this means that we are only summing over $\omega_1$ and $\omega_1'$, say. 

Observe that the $\xi$-support of $u_2^{\gamma_{12},\omega_2'}$ is contained in a tube of radius $r$, where
$$
  r \sim N_2 \gamma_{12} \sim N_1 \gamma_{12} \sim N_0\gamma_{01} \sim (N_0\Lmax^{012})^{1/2},
$$
around the axis $\R\omega_2'$. Taking the supremum over these tubes, and summing $\omega_1'$ using the Cauchy-Schwarz inequality, we get
\begin{multline}\label{D:68}
  \iint \overline{u_0}\, u_1 u_2 \, dt \, dx
  \lesssim
  \sum_{\omega_0,\omega_1}
  \left[ \min\left( N_0N_1L_1L_2, N_0^2 L_0 \Lmin^{12} \right) \right]^{1/2}
  \\
  \times
  \left( \sum_{\omega_1'}
  \norm{\Proj_{H_d(\omega_1')} u_0^{\gamma_{01},\omega_0}}^2 \right)^{1/2}
  \bignorm{u_1^{\gamma_{01},\omega_1}}
  \sup_{\omega \in \mathbb S^2} \norm{\Proj_{\R \times T_{r}(\omega)}u_2}.
\end{multline}
Recall that the sum over $\omega_1'$ is restricted by $\theta(\omega_1',\omega_1) \lesssim \gamma_{01}$. Therefore, from Lemma \ref{L:Lemma} we conclude that
\begin{equation}\label{D:72}
  \sum_{\omega_1'}
  \norm{\Proj_{H_d(\omega_1')} u_0^{\gamma_{01},\omega_0}}^2 
  \lesssim
  \left( \frac{\gamma_{01}}{\gamma_{12}} + \frac{d}{N_0\gamma_{12}^2} \right)
  \norm{u_0^{\gamma_{01},\omega_0}}^2.
\end{equation}
But since $d$ is given by \eqref{D:50},
\begin{equation}\label{D:74}
  \left( \frac{\gamma_{01}}{\gamma_{12}} + \frac{d}{N_0\gamma_{12}^2} \right)
  \sim
  \left( \frac{N_1}{N_0}
  +
  \frac{\Lmax^{12}}{N_0\gamma_{12}^2} \right)
  \sim
  \max\left( \frac{N_1}{N_0},
  \frac{N_1^2\Lmax^{12}}{N_0^2\Lmax^{012}} \right),
\end{equation}
since $\gamma_{12}^2 \sim N_0\Lmax^{012}/N_1^2$. Combining \eqref{D:68}--\eqref{D:74}, we get
\begin{multline*}
  \iint \overline{u_0}\, u_1 u_2 \, dt \, dx
  \lesssim
  \left[ \max\left( \frac{N_1}{N_0},
  \frac{N_1^2\Lmax^{12}}{N_0^2\Lmax^{012}} \right)
  \min\left( N_0N_1L_1L_2, N_0^2 L_0 \Lmin^{12} \right)\right]^{1/2}
  \\
  \times
  \sum_{\omega_0,\omega_1}
  \bignorm{u_0^{\gamma_{01},\omega_0}}
  \bignorm{u_1^{\gamma_{01},\omega_1}}
  \sup_{\omega \in \mathbb S^2} \norm{\Proj_{\R \times T_{r}(\omega)}u_2}.
\end{multline*}
Simplifying, and summing $\omega_0,\omega_1$ as in \eqref{D:38}, we get \eqref{D:30}, proving Theorem \ref{L:Thm1} in the case $\pm_1\neq\pm_2$.

\subsection{The case $\pm_1=\pm_2$} Then $\theta_{12}=\theta(\xi_1,\xi_2)$. Since $N_0 \ll N_1 \sim N_2$, Lemma \ref{I:Lemma1} implies $\theta_{12} \sim 1$ and $\Lmax^{012} \gtrsim N_1$. Applying \eqref{L:10} to $u_2$ with $N=N_2$ and $r=(N_0\Lmax^{012})^{1/2}$, and using \eqref{A:110} or \eqref{A:112}, we then get the desired estimate.

\subsection{Proof of Lemma \ref{L:Lemma}}\label{H}

The left hand side of \eqref{L:40} equals
$$
  \# \left\{ \omega \in \Omega(\gamma) \cap \Gamma_{\gamma'}(\omega_0) \colon \omega \in A \right\}, \qquad \text{where $A = \left\{ \omega \in \mathbb S^2 \colon \abs{-\tau+\xi\cdot\omega} \le d \right\}$},
$$
for given $\tau,\xi$ with $\abs{\xi} \sim N$. Without loss of generality assume $\xi = (\abs{\xi},0,0)$. Then
$$
  A \subset A' \equiv \left\{ \omega = (\omega^1,\omega^2,\omega^3) \in \mathbb S^2 \colon \omega^1 = \frac{\tau}{\abs{\xi}} + O\left(\frac{d}{N}\right) \right\}.
$$
Thus, $A'$ is the intersection of $\mathbb S^2$ and a thickened plane with normal $(1,0,0)$, and thickness comparable to $d/N$, so it looks either like a circular band or a sphere cap, which, however, can degenerate to a circle or a point, respectively. Thus,
$$
  \# \left\{ \omega \in \Omega(\gamma) \cap \Gamma_{\gamma'}(\omega_0) \colon \omega \in A' \right\}
  \lesssim 1 + \frac{\gamma'}{\gamma} + \frac{\text{area}(A')}{\gamma^2},
$$
where the first two terms cover the cases where $A'$ degenerates to a point or a circle, respectively, and the third term covers the case where $A'$ is either a sphere cap of radius $\gtrsim \gamma$ or a band of width $\gtrsim \gamma$. Since $\gamma < \gamma'$, we ignore the first term.

Using spherical coordinates we find $\text{area}(A') \lesssim d/N$, and the proof is complete.

\section{Proof of the main anisotropic estimate}\label{K}

Here we prove Theorem \ref{Z:Thm1}. By tiling (as in the proof of Lemma \ref{B:Lemma2}), it suffices to prove, given any $\delta >0$ and intervals $I_1,I_2 \subset \R$ with $\abs{I_1} = \abs{I_2} = \delta$, that
\begin{equation}\label{K:2}
  \bignorm{u_1^{I_1}u_2^{I_2}}
  \lesssim \left( \frac{\delta\Nmin^{12}L_1L_2}{\alpha} \right)^{1/2}
  \bignorm{u_1^{I_1}}
  \bignorm{u_2^{I_2}},
  \qquad \text{where $u_j^{I_j} = \Proj_{\xi_j \cdot \omega \in I_j} u_j$},
\end{equation}
and we may assume
\begin{equation}\label{K:2:2}
  \delta \ll \Nmin^{12}\alpha,
\end{equation}
since otherwise \eqref{A:76} is already better. By duality, rewrite \eqref{K:2} as
\begin{equation}\label{K:3:1}
  \iint \overline{u_0}\, u_1^{I_1} u_2^{I_2} \, dt \, dx
  \lesssim \left( \frac{\delta \Nmin^{12} L_1 L_2}{\alpha} \right)^{1/2}
  \bignorm{u_0}
  \bignorm{u_1^{I_1}}
  \bignorm{u_2^{I_2}},
\end{equation}
where $u_0 \in L^2(\R^{1+3})$ and we assume $\widehat{u_j} \ge 0$ for $j=0,1,2$. By Lemma \ref{D:Lemma2},
\begin{equation}\label{K:3:2}
  \text{l.h.s.}\eqref{K:3:1}
  \sim
  \sum_{\gamma} \sum_{\omega_1,\omega_2}
  \gamma
  \iint
  \overline{u_0} \, u_1^{I_1;\gamma,\omega_1}
  u_2^{I_2;\gamma,\omega_2} \, dt \, dx,
\end{equation}
where the sum is over dyadic $\gamma$ and $\omega_1,\omega_2 \in \Omega(\gamma)$ satisfying
\begin{equation}\label{K:3:4}
  0 < \gamma \le \frac{\pi}{1000},
  \qquad
  16\gamma \le \theta(\omega_1,\omega_2) \le M\gamma,
\end{equation}
where $M=2+36000/\pi$. For convenience we replace $\alpha$ by $2\alpha$. Splitting the support of $\widehat{u_1}$ into two symmetric parts, we may assume
\begin{equation}\label{K:4}
  \supp \widehat{u_1} \subset A_0 \equiv \left\{ (\tau,\xi) \colon
  \theta(\pm_1\xi,\omega) \le \frac{\pi}{2} - 2\alpha \right\}.
\end{equation}
Next, split the support of $\widehat{u_2}$ into three parts, by intersecting with
\begin{align*}
  A_1 &= \left\{ (\tau,\xi) \colon
  \theta(\xi,\omega^\perp) \le \alpha \right\},
  \\
  A_2 &= \left\{ (\tau,\xi) \colon
  \theta(\pm_2\xi,\omega) \le \frac{\pi}{2} - \alpha \right\},
  \\
  A_3 &= \left\{ (\tau,\xi) \colon
  \theta(\pm_2\xi,-\omega) \le \frac{\pi}{2} - \alpha \right\},
\end{align*}
whose union is $\R^{1+3}$. Correspondingly we split the proof into three cases.

\subsection{The case $\supp \widehat{u_2} \subset A_1$}\label{K:8}

Then $\gamma \ge \alpha/2$ in the sum in \eqref{K:3:2}, so since
$$
  \sum_{\genfrac{}{}{0pt}{1}{\alpha/2 \le \gamma < 1}{\text{$\gamma$ dyadic}}}
  \frac{1}{\gamma^{1/2}} \; \sim \; \frac{1}{\alpha^{1/2}},
$$
and since we can sum $\omega_1,\omega_2$ as in \eqref{D:38}, we conclude that it suffices to prove
\begin{equation}\label{K:8:2}
  \bignorm{u_1^{I_1;\gamma,\omega_1}
  u_2^{I_2;\gamma,\omega_2}}
  \lesssim \left( \frac{\delta \Nmin^{12} L_1 L_2}{\gamma} \right)^{1/2}
  \bignorm{u_1^{I_1;\gamma,\omega_1}}
  \bignorm{u_2^{I_2;\gamma,\omega_2}},
\end{equation}
under the assumption $\supp \widehat{u_2} \subset A_1$ and the separation assumption \eqref{K:3:4}.

Replacing $\gamma$ by $4\gamma$, we may assume without loss of generality that
\begin{equation}\label{K:8:4}
  \omega_2 \in \omega^\perp,
\end{equation}
while still maintaining adequate separation:
\begin{equation}\label{K:8:6}
  \theta(\omega_1,\omega_2) \ge 3\gamma.
\end{equation}
Indeed, $\supp \widehat{u_2} \subset A_1$ and $\gamma \ge \alpha/2$ imply $\theta(\omega_2,\omega^\perp) \le 3\gamma$ (or $u_2^{I_2;\gamma,\omega_2}$ vanishes), so if we rotate $\omega_2$ through this angle to get $\omega_2' \in \omega^\perp$, and replace $\gamma$ by $\gamma' = 3\gamma + \gamma = 4\gamma$, then the new sector $\Gamma_{\gamma'}(\omega_2')$ contains the original sector $\Gamma_\gamma(\omega_2)$. Moreover, the $\gamma'$-sectors around $\omega_1,\omega_2'$ are well-separated, since $\theta(\omega_1,\omega_2') \ge 16\gamma-3\gamma \ge 3\gamma'$. Dropping the primes on $\omega_2'$ and $\gamma'$, we thus have \eqref{K:8:4} and \eqref{K:8:6}.

By Lemma \ref{B:Lemma1}, we reduce to proving that for any $(\tau_0,\xi_0) \in \R^{1+3}$, the set
$$
  E = \left\{ (\tau,\xi) : \xi \cdot \omega \in I_1 \right\} \cap K^{\pm_1}_{N_1,L_1,\gamma,\omega_1} 
  \cap A_0 \cap \left( (\tau_0,\xi_0) - A_1 \cap K^{\pm_2}_{N_2,L_2,\gamma,\omega_2} \right),
$$
where we use the notation from \eqref{D:42}, verifies the volume bound
\begin{equation}\label{K:10}
  \abs{E} \lesssim \frac{\delta \Nmin^{12} L_1 L_2}{\gamma}
\end{equation}
For this, we use the same general argument as in \cite[Lemma 7.1]{Tao:2001}. Clearly,
\begin{equation}\label{K:12}
  E
  \subset
  \bigl\{ (\tau,\xi) \colon \xi \in R,  \; -\tau\pm_1\abs{\xi} = O(L_1), \; -(\tau_0-\tau)\pm_2\abs{\xi_0-\xi} = O(L_2) \bigr\},
\end{equation}
where
\begin{multline*}
  R = \Bigl\{ \xi \colon  \; \abs{\xi} \sim N_1,
  \;
  \abs{\xi_0-\xi} \sim N_2,
  \;
  \xi \cdot \omega \in I_1,
  \;
  \theta(e_1,\omega_1) \le \gamma,
  \;
  \theta(e_2,\omega_2) \le \gamma,
  \\
  \theta(e_1,\omega^\perp) \ge 2\alpha,
  \;
  \theta(e_2,\omega^\perp) \le \alpha \Bigr\}
\end{multline*}
and we use the shorthand
\begin{equation}\label{K:24}
  e_1 = \pm_1\frac{\xi}{\abs{\xi}},
  \qquad
  e_2 = \pm_2\frac{\xi_0-\xi}{\abs{\xi_0-\xi}}.
\end{equation}
We assume $\gamma \gtrsim \alpha$, as otherwise $R$ would be empty, in view of the fact that $\theta(\omega_1,\omega_2) \sim \gamma$. Integration in $\tau$ yields, using Fubini's theorem,
\begin{equation}\label{K:16}
  \abs{E} \lesssim \Lmin^{12} \Abs{ \left\{ \xi \in R \colon f(\xi) = \tau_0 + O\bigl( \Lmax^{12} \bigr) \right\} },
\end{equation}
where
\begin{equation}\label{K:20}
  f(\xi) = \pm_1\abs{\xi}\pm_2\abs{\xi_0-\xi}.
\end{equation}
Let $\xi \in R$. Then
\begin{equation}\label{K:28}
  \theta(e_1,\omega_1) \le \gamma,
  \qquad
  \theta(e_2,\omega_2) \le \gamma,
  \qquad
  \theta(e_1,\omega_2) \ge 2\gamma,
\end{equation}
where the last inequality follows by writing, using also \eqref{K:8:6},
$$
  3\gamma \le \theta(\omega_1,\omega_2) \le
  \theta(\omega_1,e_1) + \theta(e_1,\omega_2)
  \le \gamma + \theta(e_1,\omega_2).
$$
Choose coordinates $(\xi^1,\xi^2,\xi^3)$ so that $\omega = (0,0,1)$ and $\omega_2 = (1,0,0)$ (we can do this in view of \eqref{K:8:4}). Then for all $\xi \in R$, noting that
$$
  \nabla f(\xi) = e_1 - e_2,
$$
we have
\begin{equation}\label{K:30}
  - \partial_1 f(\xi)
  = \cos\theta(e_2,\omega_2) - \cos\theta(e_1,\omega_2)
  \ge \cos \gamma - \cos 2\gamma
  \gtrsim \gamma^2,
\end{equation}
where we used \eqref{K:28}. Note also that $\abs{\xi^2} \lesssim \Nmin^{12}\gamma$ on $R$, since $R$ is within an angle comparable to $\gamma$ of the $\xi^1$-axis, and inside a ball of radius comparable to $\Nmin^{12}$ around the origin. Integrating next in the $\xi^1$-direction, and using Fubini's theorem and \eqref{K:30}, we therefore get from \eqref{K:16} that
$$
  \abs{E} \lesssim \Lmin^{12} \frac{\Lmax^{12}}{\gamma^2} \Abs{ \left\{ (\xi^2,\xi^3) \colon \abs{\xi^2} \lesssim \Nmin^{12}\gamma, \; \xi^3 \in I_1  \right\} }
  \lesssim \Lmin^{12} \frac{\Lmax^{12}}{\gamma^2} \, \delta \Nmin^{12}\gamma,
$$
proving \eqref{K:10}. This concludes the proof of Theorem \ref{Z:Thm1} for $\supp\widehat{u_2} \subset A_1$.

\subsection{The case $\supp \widehat{u_2} \subset A_2$}\label{K:38}

We claim that in this case, assuming also \eqref{K:3:4},
\begin{align}\label{K:50}
  \bignorm{u_1^{I_1;\gamma,\omega_2}
  u_2^{I_2;\gamma,\omega_2}}
  &\lesssim \left( \frac{\delta \Nmin^{12} L_1 L_2}{\alpha} \right)^{1/2}
  \bignorm{u_1^{I_1;\gamma,\omega_1}}
  \bignorm{u_2^{I_2;\gamma,\omega_2}},
  \\
  \label{K:52}
  \bignorm{u_1^{I_1;\gamma,\omega_2}
  u_2^{I_2;\gamma,\omega_2}}
  &\lesssim \left( \frac{\delta (\Nmin^{12}\gamma)^2 \Lmin^{12}}{\alpha} \right)^{1/2}
  \bignorm{u_1^{I_1;\gamma,\omega_1}}
  \bignorm{u_2^{I_2;\gamma,\omega_2}}.
\end{align}
The latter holds for $\supp \widehat{u_2} \subset A_2 \cup A_3$, in fact, and does not rely on\eqref{K:3:4}.

Granting the claim for the moment, note that the part of \eqref{K:3:2} where
$$
  0 < \gamma \lesssim \gamma_0 \equiv \left(\frac{\Lmax^{12}}{\Nmin^{12}}\right)^{1/2},
$$
we can dominate by, using \eqref{K:52} and summing $\omega_1,\omega_2$ as in \eqref{D:38}, 
$$
  \sum_{0 < \gamma \lesssim \gamma_0}\gamma
  \left( \frac{\delta (\Nmin^{12})^2 \Lmin^{12}}{\alpha} \right)^{1/2} 
  \bignorm{u_0}\bignorm{u_1^{I_1}}\bignorm{u_2^{I_2}},
$$
and since $\sum_{0 < \gamma \lesssim \gamma_0}\gamma \sim \gamma_0$, we get \eqref{K:3:1}.

It remains to consider
\begin{equation}\label{K:64}
  \gamma_0 \ll \gamma < \frac{\pi}{1000}.
\end{equation}
If we argue as above, this time using \eqref{K:50}, we get \eqref{K:3:1} up to a factor $\log 1/\gamma_0$, but we can avoid this logarithmic loss by exploiting orthogonality, as we now show.

Let $X_0=X_1+X_2$ be the bilinear interaction for the summand of \eqref{K:3:2}. By \eqref{K:64}, $\Nmin^{12}\gamma^2 \gg \Lmax^{12}$, and by \eqref{K:3:4} we have $\theta_{12} \sim \gamma$, so Lemma \ref{I:Lemma1} implies
\begin{equation}\label{K:66}
  \abs{\hypwt_0} \equiv \bigabs{\abs{\tau_0}-\abs{\xi_0}} \sim
  \left\{
  \begin{alignedat}{2}
  &\Nmin^{12}\gamma^2& \quad &\text{if $\pm_1=\pm_2$},
  \\
  &\frac{N_1N_2\gamma^2}{\abs{\xi_0}}& \quad &\text{if $\pm_1\neq\pm_2$}.
  \end{alignedat}
  \right.
\end{equation}

It suffices to consider the cases $(\pm_1,\pm_2) = (+,+), (+,-)$.

Take first $(+,+)$. Then we proceed essentially as in the example given at the end of section 9 in \cite{Tao:2001}. Since $\theta(\xi_1,\omega_1), \theta(\xi_2,\omega_2) \le \gamma$ and $\theta(\omega_1,\omega_2) \le M \gamma$,
\begin{equation}\label{K:68}
  \theta(\xi_0,\omega_1) \le M'\gamma.
\end{equation}
where $M'=M+1$.  Combining this with \eqref{K:66}, we write the sum in \eqref{K:3:2} as
$$
  S = \sum_{\gamma} \sum_{\omega_1,\omega_2}
  \iint
  \overline{\Proj_{\abs{\hypwt_0} \sim \Nmin^{12}\gamma^2} u_0^{M'\gamma,\omega_1}}
  u_1^{I_1;\gamma,\omega_1}
  u_2^{I_2;\gamma,\omega_2} \, dt \, dx.
$$
Applying \eqref{K:50},
\begin{align*}
  S
  &\lesssim \sum_{\gamma} \sum_{\omega_1,\omega_2}
  \left(\frac{\delta\Nmin^{12}L_1L_2}{\alpha}\right)^{1/2}
  \bignorm{\Proj_{\abs{\hypwt_0} \sim \Nmin^{12}\gamma^2} u_0^{M'\gamma,\omega_1}}
  \bignorm{u_1^{\gamma,\omega_1}}
  \bignorm{u_2^{\gamma,\omega_2}}
  \\
  &\le
  \left(\frac{\delta\Nmin^{12}L_1L_2}{\alpha}\right)^{1/2} A B,
\end{align*}
where
\begin{align}
  \label{K:70}
  A^2 &=
  \sum_{\gamma} \sum_{\omega_1,\omega_2}
  \bignorm{\Proj_{\abs{\hypwt_0} \sim \Nmin^{12}\gamma^2} u_0^{M'\gamma,\omega_1}}^2
  \sim
  \sum_{\gamma} \bignorm{\Proj_{\abs{\hypwt_0} \sim \Nmin^{12}\gamma^2} u_0}^2
  \sim \norm{u_0}^2,
  \\
  \label{K:72}
  B^2 &= \sum_{\gamma} \sum_{\omega_1,\omega_2}
  \bignorm{u_1^{I_1;\gamma,\omega_1}}^2
  \bignorm{u_2^{I_2;\gamma,\omega_2}}^2
  \sim \bignorm{u_1^{I_1}}^2 \bignorm{u_2^{I_2}}^2.
\end{align}
Here we used \eqref{D:206} and Lemma \ref{D:Lemma1} to get \eqref{K:70}, and we used Lemma \ref{D:Lemma2} to get \eqref{K:72}. This completes the proof of \eqref{K:3:1} for the case $(+,+)$.

Next, consider $(+,-)$. This is trickier because $\xi_1,\xi_2$ point roughly in opposite directions for small $\gamma$, so \eqref{K:68} may fail. But \eqref{K:68} still holds if $N_1 \ll N_2$ or $N_2 \ll N_1$, and then $\abs{\hypwt_0} \sim \Nmin^{12}\gamma^2$ by \eqref{K:66}, so the above argument applies. That leaves $N_1 \sim N_2$, but then we can in effect reduce to $(+,+)$, by writing
\begin{equation}\label{K:80}
  \bignorm{u_1^{I_1}u_2^{I_2}}
  \le
  \bignorm{u_1^{I_1}}_{L^4} \bignorm{u_2^{I_2}}_{L^4}
  =
  \bignorm{u_1^{I_1}u_1^{I_1}}^{1/2} \bignorm{u_2^{I_2}u_2^{I_2}}^{1/2}.
\end{equation}
Since $\widehat{u_2}$ is supported away from $\omega^\perp$, both factors on the right hand side can be estimated by the $(+,+)$ case (or equivalently $(-,-)$ case) that we just proved.

This concludes the proof of Theorem \ref{Z:Thm1} for $\supp\widehat{u_2} \subset A_2$, up to the claimed estimates \eqref{K:50} and \eqref{K:52}, which we now prove.

\subsection{Proof of \eqref{K:50}}\label{K:100}

We reduce to proving
\begin{equation}\label{K:102}
  \abs{E} \lesssim \frac{\delta\Nmin^{12}L_1L_2}{\alpha},
\end{equation}
where $E$ satisfies \eqref{K:12} for some $(\tau_0,\xi_0)$, but now with
\begin{multline*}
  R = \Bigl\{ \xi \colon  \; \abs{\xi} \sim N_1,
  \;
  \abs{\xi_0-\xi} \sim N_2,
  \;
  \xi \cdot \omega \in I_1,
  \;
  \theta(e_1,\omega_1) \le \gamma,
  \;
  \theta(e_2,\omega_2) \le \gamma,
  \\
  \theta(e_1,\omega) \le \frac{\pi}{2} - \alpha,
  \;
  \theta(e_2,\omega) \le \frac{\pi}{2} - \alpha \Bigr\}
\end{multline*}
and $e_1,e_2$ as in \eqref{K:24}. Then \eqref{K:16} holds, with $f$ given by \eqref{K:20}. Assume
$$
  N_1 \le N_2,
$$
by symmetry. We claim that we may also assume
\begin{equation}\label{K:110}
  \omega_1 \in \Gamma_{\pi/2-\alpha}(\omega)
  = \left\{ \xi \colon \theta(\xi,\omega) \le \frac{\pi}{2}-\alpha \right\}.
\end{equation}
Indeed, suppose $\omega_1$ fails to satisfy this condition. We do know, however, that
$$
  \pm_1 R \subset \Gamma_\gamma(\omega_1) \cap \Gamma_{\pi/2-\alpha}(\omega),
$$
hence $\pm_1 R$ can be covered by sectors $\Gamma_\gamma(\omega_1')$ with $\omega_1' \in \Gamma_\gamma(\omega_1) \cap \Gamma_{\pi/2-\alpha}(\omega)$, and the number of such sectors required is clearly $O(1)$. Thus, we can without loss of generality assume \eqref{K:110}. The sectors are still well-separated after this change:
\begin{equation}\label{K:114}
  \theta(\omega_1,\omega_2) \ge 15\gamma,
\end{equation}
since originally we had $\theta(\omega_1,\omega_2) \ge 16\gamma$.

Choose coordinates $(\xi^1,\xi^2,\xi^3)$ so that $\omega = (0,0,1)$ and (using \eqref{K:110})
\begin{equation}\label{K:116}
  \omega_1 = (\cos\beta,0,\sin\beta),
  \qquad \text{for some $\alpha \le \beta \le \frac{\pi}{2}$}.
\end{equation}
Let $\xi \in R$. Then
\begin{equation}\label{K:118}
  \theta(e_1,\omega_1) \le \gamma,
  \qquad
  \theta(e_2,\omega_1) \ge 14\gamma,
\end{equation}
where we used \eqref{K:114} to get the last inequality. Moreover,
\begin{equation}\label{K:120}
  e_1^3, e_2^3 \ge \sin \alpha.
\end{equation}
and from \eqref{K:116} and \eqref{K:118} we see that
\begin{equation}\label{K:124}
  a \equiv \cos(\beta+\gamma) \le e_1^1 \le b \equiv \cos (\beta-c\gamma),
  \qquad
  \abs{e_1^2} \le \sin\gamma,
\end{equation}
for some $c \in [0,1]$ ($c=1$ if $\beta-\gamma \ge \alpha$, otherwise $c=(\beta-\alpha)/\gamma$).

By \eqref{K:118}, $e_2$ lies outside a disk on $\mathbb S^2$ of radius $14\gamma$ around $\omega_1$. To simplify the geometry, we want to replace this disk by a slightly smaller set which projects onto a rectangle in the $(\xi^1,\xi^2)$-plane. To this end, we apply the following lemma:

\begin{lemma}\label{K:Lemma1}
Consider a disk $D \subset \mathbb S^2$ of radius $\theta$ around $\omega_1$, where $\omega_1$ is given by \eqref{K:116} for some $\beta \in (0,\pi/2]$.

\begin{enumerate}
\item\label{K:Lemma1a} Given $0 < x \le \sin \theta$, define
$$
  y = \sin\beta \sqrt{ \sin^2 \theta  - x^2 }.
$$
Then $D$ contains the subset of $\mathbb S^2_+ = \{ e \in \mathbb S^2 : e^3 >  0 \}$ whose projection onto the $(\xi^1,\xi^2)$-plane is the intersection of
$$
  R = \left\{ (\xi^1,\xi^2) \colon \abs{\xi^1-\cos\beta\cos\theta} \le y, \;\; \abs{\xi^2} \le x \right\}
$$
and the unit disk $(\xi^1)^2 + (\xi^2)^2 < 1$.

\item\label{K:Lemma1b} Suppose further that $\beta < \theta \le \pi/2$, so that the disk $D$ dips below the equator, i.e., the boundary of $\mathbb S^2_+$. Define
$$
  x = \sqrt{ 1 - \frac{\cos^2\theta}{\cos^2\beta}},
  \qquad
  y = \sin\beta \sqrt{ \sin^2 \theta  - x^2 }.
$$
Then $D$ contains the subset of $\mathbb S^2_+$ whose projection onto the $(\xi^1,\xi^2)$-plane is the intersection of
$$
  R' = \left\{ (\xi^1,\xi^2) \colon \xi^1 \ge \cos\beta\cos\theta - y,
  \;\;
  \abs{\xi^2} \le x \right\}
$$
and the unit disk $(\xi^1)^2 + (\xi^2)^2 < 1$.
\end{enumerate}
\end{lemma}

The proof is given in section \ref{G}.

Applying part \eqref{K:Lemma1a} of the lemma, with $\theta=14\gamma$, $x=\sin4\gamma$ and
\begin{equation}\label{K:128}
  y = \sin\beta \sqrt{ \sin^2 14\gamma - \sin^2 4\gamma },
\end{equation}
we conclude from \eqref{K:118} that
\begin{equation}\label{K:130}
  \Abs{e_2^2} \ge \sin 4\gamma,
\end{equation}
or
\begin{equation}\label{K:132}
  e_2^1 \notin [a',b'],
  \qquad
  a'=\cos\beta \cos 14\gamma - y,
  \qquad
  b'=\cos\beta \cos 14\gamma  + y.
\end{equation}

First suppose \eqref{K:130} holds for some $\xi \in R$. Since the angle between any two $e_2$'s is no larger than $2\gamma$, it follows that $\Abs{e_2^2} \ge \sin 2\gamma$ for all $\xi \in R$, so by \eqref{K:124},
$$
  \Abs{\partial_2 f(\xi)} = \Abs{e_1^2 - e_2^2}
  \ge \sin 2\gamma - \sin\gamma \gtrsim \gamma
  \qquad (\forall \xi \in R),
$$
Thus, integrating next in the $\xi^2$-direction, we see from \eqref{K:16} that
\begin{equation}\label{K:134}
  \Abs{E} 
  \lesssim \Lmin^{12} \frac{\Lmax^{12}}{\gamma} \Abs{P_{(\xi^1,\xi^3)}(R)},
\end{equation}
where $P_{(\xi^1,\xi^3)}$ is the projection onto the $(\xi^1,\xi^3)$-plane. But clearly,
\begin{equation}\label{K:136}
  \Abs{P_{(\xi^1,\xi^3)}(R)} \lesssim \frac{r N_1\gamma}{\beta},
\end{equation}
since, by our choice of coordinates, $\xi^3$ is restricted to an interval of length $r$, and $\xi$ lies within an angle $\gamma$ of $\omega_1 = (\cos\beta,0,\sin\beta)$ and at distance $\sim N_1$ from the origin. By \eqref{K:134} and \eqref{K:136}, we get \eqref{K:102} for the case where \eqref{K:130} holds for some $\xi \in R$.

It remains to consider the case where \eqref{K:132} holds for all $\xi \in R$. We shall use
\begin{equation}\label{K:138}
  (1-\varepsilon) \theta \le \sin \theta \le \theta  \qquad \text{for $0 < \theta \le \frac{14\pi}{1000}$, where $\varepsilon = 10^{-3}$}.
\end{equation}
Thus, $y \ge 13 \gamma \sin\beta$, hence
\begin{equation}\label{K:150}
  a-a'
  \ge \cos\beta \left( \cos\gamma - \cos 14\gamma \right)
  + \sin\beta \left( 13\gamma - \sin\gamma \right)
  \ge
  12\gamma \sin\beta.
\end{equation}
Moreover, using also the fact that $1-\cos\theta \le \theta^2/2$ for all $\theta$,
\begin{align*}
  b'-b
  &\ge - \cos\beta \left( \cos c\gamma - \cos 14\gamma \right)
  + \sin\beta \left( 13\gamma - \sin c \gamma \right)
  \\
  &\ge -\left( 1 - \cos 14\gamma \right)
  + \sin\beta \left( 13\gamma - \gamma \right)
  \\
  &\ge - \frac{(14\gamma)^2}{2}
  +  12\gamma \sin\beta 
  \ge
  12\gamma \left( \sin\beta-9\gamma \right)
  \ge
  12\gamma \left( \sin\beta-\sin 10\gamma \right),
\end{align*}
which implies
\begin{equation}\label{K:154}
  b'-b \gtrsim \beta\gamma \qquad \text{if $\beta \ge 11\gamma$},
\end{equation}
hence it is natural to split into the cases $\beta \ge 11\gamma$ and $\beta < 11\gamma$.

Assume first $\beta \ge 11\gamma$. Then by \eqref{K:124}, \eqref{K:132}, \eqref{K:150} and \eqref{K:154},
$$
  \Abs{\partial_1 f(\xi)} = \Abs{e_1^1 - e_2^1}
  \ge \min(a-a', b'-b) \gtrsim \beta\gamma
$$
for all $\xi \in R$, so integrating next in the $\xi^1$-direction we get
$$
  \Abs{E} 
  \lesssim \Lmin^{12} \frac{\Lmax^{12}}{\beta\gamma} \Abs{P_{(\xi^2,\xi^3)}(R)},
$$
But by our choice of coordinates,
$$
  \Abs{P_{(\xi^2,\xi^3)}(R)}
  \le
  \Abs{ \left\{ (\xi^2,\xi^3) \colon
  \abs{\xi^2} \lesssim N_1\gamma, \; \xi^3 \in I_1 \right\} }
  \lesssim \delta N_1\gamma,
$$
so we get \eqref{K:102}, recalling that $\beta \ge \alpha$.

Next, consider
\begin{equation}\label{K:160}
  \beta < 11\gamma.
\end{equation}
Then we use part \eqref{K:Lemma1b} of Lemma \ref{K:Lemma1}, concluding that $e_2$ must satisfy
\begin{equation}\label{K:162}
  \Abs{e_2^2} \ge x \equiv \sqrt{1-\frac{\cos^2 14\gamma}{\cos^2\beta}}
  = \frac{\sqrt{\sin^2 14\gamma - \sin^2\beta}}{\cos\beta},
\end{equation}
or
\begin{equation}\label{K:164}
  e_2^1 \le a'' \equiv \cos\beta \cos 14\gamma - \sin\beta \sqrt{ \sin^2 14\gamma - x^2 }.
\end{equation}
By \eqref{K:138} and \eqref{K:160},
$x \ge \sqrt{(1-\varepsilon)^2(14\gamma)^2 - (11\gamma)^2} \ge 8\gamma$, so \eqref{K:162} is stronger than \eqref{K:130}, hence we know how to deal with it. This leaves the case where \eqref{K:164} holds for all $\xi \in R$. By \eqref{K:160}, $\beta < 11\pi/1000$, hence $\cos\beta \ge 1-\varepsilon$, where $\varepsilon = 10^{-3}$, so
\begin{align*}
  a-a''
  &= \cos\beta(\cos\gamma - \cos 14\gamma)
  + \sin\beta \left( \sqrt{ \sin^2 14\gamma - x^2 } - \sin\gamma \right)
  \\
  &\ge (1-\varepsilon)(\cos\gamma - \cos 14\gamma) 
  - \sin\beta\sin\gamma
  \\
  &\ge (1-\varepsilon)13\gamma\sin\gamma 
  - 11\gamma^2
  \ge (1-\varepsilon)^213\gamma^2 
  - 11\gamma^2 \ge \gamma^2 \ge \frac{\beta\gamma}{11},
\end{align*}
which replaces \eqref{K:150}, hence we can integrate in the $\xi^1$-direction.

This concludes the proof of \eqref{K:50}.

\subsection{Proof of \eqref{K:52}}\label{K:200}

Assuming $\supp\widehat{u_2} \subset A_2 \cup A_3$, but not \eqref{K:3:4}, we need
$$
  \abs{E} \lesssim \frac{\delta(\Nmin^{12}\gamma)^2 \Lmin^{12}}{\alpha},
$$
where $E$ satisfies \eqref{K:12} for some $(\tau_0,\xi_0)$, but now with
\begin{multline*}
  R = \Bigl\{ \xi \colon  \; \abs{\xi} \sim N_1,
  \;
  \abs{\xi_0-\xi} \sim N_2,
  \;
  \xi \cdot \omega \in I_1,
  \;
  \theta(e_1,\omega_1) \le \gamma,
  \;
  \theta(e_2,\omega_2) \le \gamma,
  \\
  \theta(e_1,\omega) \le \frac{\pi}{2} - \alpha,
  \;
  \theta(e_2,\R\omega) \le \frac{\pi}{2} - \alpha
  \Bigr\}.
\end{multline*}
By symmetry, we may assume $N_1 \le N_2$, and then we simplify to
$$
  R = \left\{ \xi \colon  \; \abs{\xi} \sim N_1,
  \;
  \xi \cdot \omega \in I_1,
  \;
  \theta(e_1,\omega_1) \le \gamma,
  \; 
  \theta(e_1,\omega^\perp) \ge \alpha
  \right\}.
$$
Integrating $\tau$ yields $\Abs{E} \lesssim \Lmin^{12} \Abs{R}$, so it suffices to show $ \Abs{R} \lesssim \delta(N_1\gamma)^2/\alpha$, but this is easy; we omit the details.

\subsection{The case $\supp \widehat{u_2} \subset A_3$}\label{K:298} The trick \eqref{K:80} takes care of the case $N_1 \sim N_2$, effectively reducing to $\supp \widehat{u_2} \subset A_2$. Thus, it suffices to consider, by symmetry,
\begin{equation}\label{K:300}
  N_1 \ll N_2.
\end{equation}
Now we repeat the argument from subsection \ref{K:38}. We know that \eqref{K:52} is valid, so we just need to show that \eqref{K:50} holds, under the additional assumption \eqref{K:300}. Again we reduce to proving \eqref{K:102}, but now with
\begin{multline*}
  R = \Bigl\{ \xi \colon  \; \abs{\xi} \sim N_1,
  \;
  \abs{\xi_0-\xi} \sim N_2,
  \;
  \xi \cdot \omega \in I_1,
  \;
  \theta(e_1,\omega_1) \le \gamma,
  \;
  \theta(e_2,\omega_2) \le \gamma,
  \\
  \theta(e_1,\omega) \le \frac{\pi}{2} - \alpha,
  \;
  \theta(e_2,-\omega) \le \frac{\pi}{2} - \alpha \Bigr\}.
\end{multline*}
We assume $R \neq \emptyset$, hence $\gamma \gtrsim \alpha$. For $\xi \in R$, $\abs{\nabla f(\xi)} = \abs{e_1-e_2} \sim \gamma$, but $e_1,e_2$ can be symmetrically placed about the $(\xi^1,\xi^2)$-plane, hence $(\partial_1 f,\partial_2 f)$ may vanish, and then we have no choice but to integrate in the direction $\xi^3$. Since
\begin{equation}\label{K:314}
  \partial_3 f(\xi) = e_1^3 - e_2^3 \ge 2\sin\alpha 
  \qquad (\forall \xi \in R),
\end{equation}
we then get, from \eqref{K:16},
$$
  \Abs{E} \lesssim \frac{L_1L_2}{\alpha} \Abs{ P_{(\xi^1,\xi^2)}(B \cap S \cap \{ \xi \colon \xi^3 \in I \} )},
$$
where
$$
  B = \left\{ \xi \colon \abs{\xi} \lesssim N_1 \right\},
  \qquad
  S = \left\{ \xi \colon f(\xi) = \tau_0 + O\bigl( \Lmax^{12} \bigr) \right\},
$$
and $f$ is given by \eqref{K:20}. Thus, it will be enough to show that
\begin{equation}\label{K:320}
  \Abs{ P_{(\xi^1,\xi^2)}(B \cap S \cap \{ \xi \colon \xi^3 \in I \})}
  \lesssim \delta N_1.
\end{equation}
Of course, this may fail if $\Lmax^{12}$ is large, but by the argument in Remark \ref{C:Rem}, we can assume $L_1,L_2 > 0$ arbitrarily small in \eqref{K:50}. In particular, we may assume
\begin{equation}\label{K:324}
  \Lmax^{12} \ll N_1\gamma^2.
\end{equation}
Then by Lemma \ref{I:Lemma1},
\begin{equation}\label{K:328}
  \bigabs{\abs{\tau_0} - \abs{\xi_0}} \sim N_1\gamma^2,
  \qquad \abs{\tau_0} \sim \abs{\xi_0} \sim N_2,
\end{equation}
where we also used \eqref{K:300}, which implies $\abs{\xi_0} \sim N_2$.

The set $S$ is a thickening of the surface
$$
  S_0 = \left\{ \xi \colon f(\xi) = \tau_0 \right\},
$$
which is an ellipsoid if $\pm_1=\pm_2$, or one sheet of a hyperboloid if $\pm_1\neq\pm_2$, both with foci at $0$ and $\xi_0$, and rotationally symmetric about the axis through the foci. The major and minor semiaxes of $S_0$, denoted $a$ and $b$, respectively, are given by
\begin{equation}\label{K:340}
  2a = \abs{\tau_0} \sim N_2,
  \qquad
  2b = \bigabs{\tau_0^2 - \abs{\xi_0}^2}^{1/2} \sim (N_1N_2)^{1/2} \gamma.
\end{equation}
Since $\abs{\nabla f(\xi)} = \abs{e_1-e_2} = 2b(\abs{\xi}\abs{\xi_0-\xi})^{-1/2}$ for $\xi \in S_0$, we have
\begin{equation}\label{K:342}
  \abs{\nabla f(\xi)}
  \gtrsim \gamma \qquad \text{on $B \cap S_0$},
\end{equation}
hence $B \cap S$ is contained in an $\varepsilon$-neighborhood of $S_0$, where $\varepsilon \sim \Lmax^{12}/\gamma$. But we can assume $\varepsilon$ arbitrarily small, since $\Lmax^{12}$ is arbitrarily small. So let us assume
\begin{equation}\label{K:346}
  \varepsilon \ll \delta, \qquad \varepsilon \ll N_1\gamma^2.
\end{equation}

The minimal radius of curvature on $S_0$, which we denote $R_*$, satisfies
$$
  R_* \sim \frac{b^2}{a} \sim N_1\gamma^2,
$$
so the second inequality in \eqref{K:346} guarantees that the $\varepsilon$-neighborhood of $S_0$ is in fact a tubular neighborhood.

Let the interval $I^*$ have the same center as $I$ but twice the length. Then for any $p \in S_0 \cap \{ \xi \colon \xi^3 \in I \}$, there is a disk $D \subset S_0$ centered at $p$ and of radius $r$, such that
$$
  r \sim \min\left(\delta,\sqrt{R_* \delta}\right) \gg \varepsilon,
  \qquad D \subset S_0 \cap \{ \xi \colon \xi^3 \in I^* \}.
$$
Thus, $S_0 \cap \{ \xi \colon \xi^3 \in I^* \}$ cannot be ``narrower'' than $\varepsilon$ anywhere.

Let $M$ be the number of $\varepsilon$-cubes $Q$ needed to cover
$$
  S' = B \cap S \cap \{ \xi \colon \xi^3 \in I \}.
$$
Since the $\varepsilon$-neighborhood of $S_0$ is tubular, it suffices to consider cubes centered on $S_0$, and since $S_0 \cap \{ \xi \colon \xi^3 \in I^* \}$ cannot be ``narrower'' than $\varepsilon$, we conclude that
$$
  M = O\left( \frac{A}{\varepsilon^2} \right),
  \qquad
  \text{where}
  \qquad
  A = \sigma(B \cap S_0 \cap \{ \xi \colon \xi^3 \in I^* \}).
$$
Therefore, the area of the projection of $S'$ onto any plane in $\R^3$ is $O(\varepsilon^2 M) = O(A)$, and this proves \eqref{K:320} provided that we can show
$$
  A \lesssim N_1 \delta,
$$
but this holds by Theorem \ref{Z:Thm2}, since $b^2/a \sim N_1\gamma^2 \ll N_2 \sim a$.

This concludes the proof of Theorem \ref{Z:Thm1}.

\section{Proofs of the null form estimates}\label{E}

\subsection{Proof of Theorem \ref{N:Thm2}}\label{E:2}

By duality, write the estimate in Theorem \ref{N:Thm2} as
\begin{equation}\label{E:24}
  \iint \overline{u_0}\, \nullform(\Proj_{\R \times T_r(\omega)} u_1,u_2) \, dt \, dx
  \lesssim \left( r^2 L_1 L_2 \right)^{1/2}
  \norm{u_0}
  \norm{u_1}
  \norm{u_2},
\end{equation}
where $u_0,u_1,u_2 \in L^2(\R^{1+3})$ and without loss of generality $\widehat{u_j} \ge 0$ for $j=0,1,2$. As usual, $u_1,u_2$ are assumed to satisfy \eqref{A:120}. By Lemma \ref{D:Lemma2},
\begin{equation}\label{E:18}
  \text{l.h.s.}\eqref{E:24}
  \sim
  \sum_{\gamma} \sum_{\omega_1,\omega_2}
  \gamma
  \iint
  \overline{u_0}
  \left( \Proj_{\R \times T_r(\omega)} u_1^{\gamma,\omega_1} \right)
  u_2^{\gamma,\omega_2} \, dt \, dx,
\end{equation}
where the sum is over dyadic $0 < \gamma < 1$ and $\omega_1,\omega_2 \in \Omega(\gamma)$ with
\begin{equation}\label{E:19}
  3\gamma \le \theta(\omega_1,\omega_2) \le 12\gamma.
\end{equation}
We claim that
\begin{align}
  \label{E:20}
  \norm{\Proj_{\R \times T_r(\omega)} u_1^{\gamma,\omega_1}
  \cdot u_2^{\gamma,\omega_2}}
  &\lesssim \left( \frac{r^2 L_1 L_2}{\gamma^2} \right)^{1/2}
  \norm{u_1^{\gamma,\omega_1}}
  \norm{u_2^{\gamma,\omega_2}},
  \\
  \label{E:22}
  \norm{\Proj_{\R \times T_r(\omega)} u_1^{\gamma,\omega_1}
  \cdot u_2^{\gamma,\omega_2}}
  &\lesssim \left( r^2 \Nmin^{12} \Lmin^{12} \right)^{1/2}
  \norm{u_1^{\gamma,\omega_1}}
  \norm{u_2^{\gamma,\omega_2}},
  \\
  \label{E:21}
  \norm{\Proj_{\R \times T_r(\omega)} u_1^{\gamma,\omega_1}
  \cdot u_2^{\gamma,\omega_2}}
  &\lesssim \left( (\Nmin^{12})^2 L_1 L_2 \right)^{1/2}
  \norm{u_1^{\gamma,\omega_1}}
  \norm{u_2^{\gamma,\omega_2}}.
\end{align}
The first two are proved in subsection \ref{E:38:2}; the last one hods by Theorem \ref{A:Thm1}.

Arguing as in subsection \ref{K:38}, and using either \eqref{E:22} or \eqref{E:21}, we get the desired estimate for that part of \eqref{E:18} which corresponds to
\begin{equation}\label{E:23}
  0 < \gamma \lesssim \gamma_0 \equiv \max \left( 
  \left(\frac{\Lmax^{12}}{\Nmin^{12}}\right)^{1/2},
  \frac{r}{\Nmin^{12}} \right),
\end{equation}
so for the remainder of this subsection we restrict the sum to
\begin{equation}\label{E:25}
  \gamma_0 \ll \gamma < 1,
\end{equation}
and then we use the estimate \eqref{E:20}. To avoid a logaritmic loss we argue as in subsection \ref{K:38}, but use also the  fact that since $\xi_1 \in T_r(\omega)$ and $\abs{\xi_1} \sim N_1$, we may assume (replacing $\omega$ by $-\omega$ if necessary)
\begin{equation}\label{E:26:4}
  \theta(\pm_1\xi_1,\omega) \lesssim \frac{r}{N_1} \ll \gamma,
\end{equation}
where the last inequality is due to \eqref{E:25}. Moreover, $\theta(\pm_1\xi_1,\omega_1) \le \gamma$, hence
\begin{equation}\label{E:26:2}
  \theta(\omega_1,\omega) \le \frac32 \gamma,
\end{equation}
implying that $\omega_1 \in \Omega(\gamma)$ is essentially uniquely determined, hence so is $\omega_2$.

By \eqref{E:25}, $\Nmin^{12}\gamma^2 \gg \Lmax^{12}$, and by \eqref{E:19}, $\theta_{12} \sim \gamma$, hence \eqref{K:66} holds.

It suffices to consider $(\pm_1,\pm_2) = (+,+), (+,-)$. For $(+,+)$ we proceed almost exactly as in subsection \ref{K:38}, so we omit the details.

Now consider $(+,-)$. Then the argument for $(+,+)$ applies if $N_1 \ll N_2$ or $N_2 \ll N_1$, but not if $N_1 \sim N_2$. In subsection \ref{K:38} we dealt with the latter by reducing to linear estimates, which again puts us into the $(+,+)$ case. This does not work here, however, since our estimate is not symmetric. Instead, we proceed as for $(+,+)$, but use also the crucial additional fact that $\omega_1,\omega_2$ are essentially uniquely determined, as shown above. Using \eqref{K:66} we write \eqref{E:18} as
$$
  \sum_{\gamma} \sum_{\omega_1,\omega_2}
  \gamma
  \iint
  \overline{\Proj_{\abs{\hypwt_0} \sim \frac{N_1N_2\gamma^2}{\abs{\xi_0}}} u_0}
  \left( \Proj_{\R \times T_r(\omega)} u_1^{\gamma,\omega_1} \right)
  u_2^{\gamma,\omega_2} \, dt \, dx
  \lesssim
  \left(r^2L_1L_2\right)^{1/2} A B.
$$
Here we used \eqref{E:20}, $B$ is defined as in \eqref{K:72} (but without the $I_1,I_2$, of course), and
$$
  A^2 =
  \sum_{\gamma}
  \bignorm{\Proj_{\abs{\hypwt_0} \sim \frac{N_1 N_2\gamma^2}{\abs{\xi_0}}} u_0}^2,
$$
in view of \eqref{E:26:2}. Clearly, $A^2 \lesssim \norm{u_0}^2$, so we are done.

This completes the proof of Theorem \ref{N:Thm2}, up to the estimates \eqref{E:20} and \eqref{E:22}, which we prove in the following subsection.

\subsection{Proof of \eqref{E:20} and \eqref{E:22}}\label{E:38:2} First, \eqref{E:20} reduces to
\begin{equation}\label{E:40}
  \abs{E} \lesssim \frac{r^2 L_1 L_2}{\gamma^2},
\end{equation}
where $E$ satisfies \eqref{K:12} with $R$ given by
$$
  R = \bigl\{ \xi \in T_r(\omega) \colon  \abs{\xi} \sim N_1, \; \theta(e_1,\omega_1) \le \gamma, \;
  \; \theta(e_2,\omega_2) \le \gamma \bigr\},
$$
with $e_1,e_2$ as in \eqref{K:24}. Then \eqref{K:16} holds with $f$ as in \eqref{K:20}. Let $\xi \in R$. By \eqref{E:19},
\begin{equation}\label{E:46}
  \theta(e_2,\omega_1) \ge 2\gamma.
\end{equation}
We may assume $r \ll N_1\gamma$ (otherwise \eqref{E:20} holds by \eqref{E:21}) hence \eqref{E:26:4} holds, i.e., 
\begin{equation}\label{E:48}
  \theta(e_1,\omega) \lesssim \frac{r}{N_1} \ll \gamma.
\end{equation}
Since $\theta(e_2,\omega_1) \le \theta(e_2,\omega) + \theta(\omega,e_1) + \theta(e_1,\omega_1)$, \eqref{E:46} and \eqref{E:48} imply:
\begin{equation}\label{E:50}
  \theta(e_2,\omega) \ge \frac12 \gamma.
\end{equation}
Combining \eqref{E:48} and \eqref{E:50} gives, for $\xi \in R$,
$$
  \nabla f(\xi) \cdot \omega = (e_1-e_2)\cdot\omega
  = \cos \theta(e_1,\omega) - \cos \theta(e_2,\omega)
  \ge \cos\frac14\gamma-\cos\frac12 \gamma \sim \gamma^2,
$$
so integrating next in the direction $\omega$, we see from \eqref{K:16} that
$$
  \abs{E} \lesssim \Lmin^{12} \frac{\Lmax^{12}}{\gamma^2}
  \bigabs{P_{\omega^\perp}(T_r(\omega))}
  \lesssim \Lmin^{12} \frac{\Lmax^{12}}{\gamma^2}
  r^2,
$$
where $P_{\omega^\perp}$ is the projection onto $\omega_1^\perp$. This proves \eqref{E:40}.

Finally, \eqref{E:22} reduces to the trivial estimate $\abs{R} \lesssim r^2 \Nmin^{12}$, where $R$ is the set of $\xi \in T_r(\omega)$ such that $\abs{\xi} \lesssim N_1$ and $\abs{\xi_0-\xi} \lesssim N_2$, for some fixed $\xi_0$.

\subsection{Proof of Theorem \ref{N:Thm3}}

Arguing as in the proof of Theorem \ref{N:Thm2}, we reduce to proving that
\begin{equation}\label{E:52}
  \norm{\Proj_{\R \times B} u_1^{\gamma,\omega_2}
  \cdot u_2^{\gamma,\omega_2}}
  \lesssim C
  \norm{u_1^{\gamma,\omega_1}}
  \norm{u_2^{\gamma,\omega_2}}
\end{equation}
holds with $C^2 \sim r^2 L_1 L_2/\gamma$ and $C^2 \sim r^3 \Lmin^{12}$, and we further reduce to proving the corresponding volume bounds for $E$ satisfying \eqref{K:12} with
$$
  R = \bigl\{ \xi \in B \colon \theta(e_1,\omega_1) \le \gamma, \;
  \; \theta(e_2,\omega_2) \le \gamma \bigr\},
$$
where $e_1,e_2$ are defined as in \eqref{K:24}. In view of \eqref{E:19}, $\Gamma_\gamma(\omega_1), \Gamma_\gamma(\omega_2)$ are $\gamma$-separated, so it is clearly possible to choose the coordinates $(\xi^1,\xi^2,\xi^3)$ so that $e_1^1-e_2^1 \sim \gamma$ for all $\xi \in R$. Thus, $\partial_1 f \sim \gamma$ on $R$, so from \eqref{K:16} we get
$$
  \abs{E} \lesssim \Lmin^{12} \frac{\Lmax^{12}}{\gamma}
  \Abs{ \left\{(\xi^2,\xi^3) \colon \abs{\xi^2-\xi_*^2} \lesssim r,
  \; \abs{\xi^3-\xi_*^3} \lesssim r \right\} }
  \lesssim \Lmin^{12} \frac{\Lmax^{12}}{\gamma} r^2,
$$
where $\xi_*$ is the center of $B$. This proves \eqref{E:52} with $C^2 \sim r^2 L_1 L_2/\gamma$. The other estimate reduces to $\Abs{E} \lesssim r^3 \Lmin^{12}$, but this is trivial, since $\abs{R} \lesssim r^3$.

This concludes the proof of Theorem \ref{N:Thm3}.

\section{Proof of the concentration/nonconcentration estimate}\label{J}

Write the estimate in Theorem \ref{N:Thm4} in the dual form \eqref{E:24}, but with $\Proj_{\xi_0 \cdot \omega \in I_0}$ inserted in front of the null form $\nullform$. By Lemma \ref{D:Lemma2} we then reduce to proving
\begin{multline}\label{J:190}
  \sum_{\gamma} \sum_{\omega_1,\omega_2}
  \gamma
  \iint
  \overline{u_0}
  \, \Proj_{\xi_0 \cdot \omega \in I_0}
  \left( \Proj_{\R \times T_r(\omega)} u_1^{\gamma,\omega_1} \right)
  u_2^{\gamma,\omega_2} \, dt \, dx
  \\
  \lesssim
  \left(r^2L_1L_2\right)^{1/2}
  \norm{u_0}
  \left(\sup_{I_1} \norm{\Proj_{\xi_1 \cdot \omega \in I_1} u_1}\right)
  \norm{u_2}
\end{multline}
where the sum is over dyadic $\gamma$  and $\omega_1,\omega_2 \in \Omega(\gamma)$ satisfying
\begin{equation}\label{J:194}
  0 < \gamma \ll 1, \qquad 3\gamma \le \theta(\omega_1,\omega_2) \le 12\gamma.
\end{equation}
For $\gamma$ satisfying \eqref{E:23}, we argue as in the proof of Theorem \ref{N:Thm2} in subsection \ref{E:2}, but instead of \eqref{E:22} and \eqref{E:21} we use the estimates (proved below)
\begin{align}
  \label{J:200}
  \norm{\Proj_{\xi_0 \cdot \omega \in I_0} \left(
  \Proj_{\R \times T_r(\omega)} u_1^{\gamma,\omega_1} \cdot u_2^{\gamma,\omega_2} \right)}
  &\lesssim \left( r^2 \abs{I_0} \Lmin^{12} \right)^{1/2} 
  \norm{u_1^{\gamma,\omega_1}}
  \norm{u_2^{\gamma,\omega_2}},
  \\
  \label{J:202}
  \norm{\Proj_{\xi_0 \cdot \omega \in I_0} \left(
  \Proj_{\R \times T_r(\omega)} u_1^{\gamma,\omega_1} \cdot u_2^{\gamma,\omega_2} \right)}
  &\lesssim \left( \abs{I_0} \Nmin^{12} L_1 L_2 \right)^{1/2} 
  \norm{u_1^{\gamma,\omega_1}}
  \norm{u_2^{\gamma,\omega_2}}.
\end{align}
If we also tile by the condition $\xi_0 \cdot \omega \in I_0$, then we see that the part of \eqref{J:190} corresponding to \eqref{E:23} is dominated by
\begin{equation}\label{J:210}
  \left(\frac{\abs{I_0}}{\Nmin^{12}}\right)^{1/2}
  \left(r^2L_1L_2\right)^{1/2}
  \sum_{I_1,I_2}
  \bignorm{u_0}
  \bignorm{u_1^{I_1}}
  \bignorm{u_2^{I_2}}
  \qquad \left( u_j^{I_j} = \Proj_{\xi_j \cdot \omega \in I_j} u_j \right),
\end{equation}
where $I_1, I_2$ belong to the almost disjoint cover of $\R$ by translates of $I_0$, and the sum is restricted by the condition $(I_1 + I_2) \cap I_0 \neq \emptyset$, hence the sum is over a set of cardinality comparable to $\Nmin^{12}/\abs{I_0}$, and each $I_1$ can interact with at most three different $I_2$'s. Thus, sup'ing over $I_1$ and summing $I_2$ using the Cauchy-Schwarz inequality, we get the bound in the right hand side of \eqref{J:190}.

Now it only remains to consider $\gamma_0 \ll \gamma \ll 1$ with $\gamma_0$ defined as in \eqref{E:23}, and then we use the estimate
\begin{equation}\label{J:220}
  \norm{\Proj_{\xi_0 \cdot \omega \in I_0} \left(
  \Proj_{\R \times T_r(\omega)} u_1^{\gamma,\omega_1} \cdot u_2^{\gamma,\omega_2} \right)}
  \lesssim \left( \frac{r^2L_1L_2}{\gamma^2} \right)^{1/2} 
  \left( \sup_{I_1} \bignorm{u_1^{I_1}} \right)
  \bignorm{u_2^{\gamma,\omega_2}},
\end{equation}
which is proved below. To avoid a logarithmic loss, we repeat the argument given at the end of subsection \ref{E:2}, the only difference being that we cannot define $B$ as in \eqref{K:72}, due to the supremum on the norm of $u_1$. Instead, $B$ is now given by
$$
  B^2 = \left(\sup_{I_1} \bignorm{u_1^{I_1}}\right)^2
  \sum_{\gamma} \sum_{\omega_1,\omega_2}
  \bignorm{u_2^{I_2;\gamma,\omega_2}}^2.
$$
But since $\gamma \gg \gamma_0$, we have $r \ll \Nmin^{12}\gamma$, hence we may assume \eqref{E:26:2}. Thus, $\omega_1 \in \Omega(\gamma)$ is essentially uniquely determined, and using also \eqref{J:194} we see that $\theta(\omega_2,\omega) \ge (3/2)\gamma$. Thus,
$$
  \sum_{\gamma} \sum_{\omega_1,\omega_2}
  \bignorm{u_2^{I_2;\gamma,\omega_2}}^2
  \sim
  \sum_{\gamma}
  \bignorm{\Proj_{\theta(\pm_2\xi_2,\omega) \sim \gamma} u_2^{I_2}}^2
  \sim \bignorm{u_2^{I_2}}^2,
$$
so the argument at the end of subsection \ref{E:2} goes through.

It now remains to prove the claimed estimates \eqref{J:200}, \eqref{J:202} and \eqref{J:220}. Then the proof of Theorem \ref{N:Thm4} will be complete.

Observe that \eqref{J:200} follows by an obvious modification of the proof of \eqref{E:22}, given at the end of subsection \ref{E:38:2}. The estimate \eqref{J:202} follows from Theorem \ref{Z:Thm1}; we can ensure that \eqref{Z:10} holds with $\alpha$ bounded away from zero, since we may assume
$$
  \theta(\omega_1,\omega) \le \gamma + O\left(\frac{r}{N_1}\right),
$$
where $\gamma \ll 1$ and $r \ll N_1$, by the hypotheses of Theorem \ref{N:Thm4}, hence $\theta(\omega_1,\omega) \ll \pi/2$.

Now it only remains to prove \eqref{J:220}, but this requires some work. We split the proof into several subsections.

\subsection{Preliminaries}

Recall that we are only claiming \eqref{J:220} under the assumption $\gamma_0 \ll \gamma \ll 1$, which in particular implies
\begin{equation}\label{J:240}
  r \ll \Nmin^{12}\gamma.
\end{equation}
We shall denote by $\phi$ the smallest angle such that
\begin{equation}\label{J:241}
  \Delta B_{N_1} \cap T_r(\omega) \subset \Gamma_\phi(\omega).
\end{equation}
Thus,
\begin{equation}\label{J:242}
  \phi \sim \frac{r}{N_1} \ll \gamma,
\end{equation}
so replacing $\omega$ by $-\omega$ if necessary, we may assume that
\begin{equation}\label{J:244}
  \Gamma_\phi(\omega) \subset \Gamma_\gamma(\omega_1).
\end{equation}
To simplify the ensuing discussion, we change our notation slightly, assuming
\begin{equation}\label{J:246}
  u_1,u_2 \in L^2(\R^{1+3}),
  \qquad
  \supp\widehat{u_1} \subset S_1,
  \qquad
  \supp\widehat{u_2} \subset S_2,
\end{equation}
where
\begin{equation}
  \label{J:248}
  S_1 = \bigl( \R \times T_{r}(\omega) \bigr)
  \cap K^{\pm_1}_{N_1,L_1,\phi,\omega},
  \qquad
  S_2 = K^{\pm_2}_{N_2,L_2,\gamma,\omega_2},
\end{equation}
We then want to prove that
\begin{equation}\label{J:250}
  \norm{\Proj_{\xi_0 \cdot \omega \in I_0} \left(u_1u_2\right)}
  \le C
  \left( \sup_{I_1} \norm{u_1^{I_1}} \right)
  \norm{u_2}
  \qquad \text{($\forall u_1,u_2$ as in $\eqref{J:246}$)}
\end{equation} 
holds with $C^2 \sim r^2 L_1 L_2/\gamma^2$, where the supremum is over all translates $I_1$ of $I_0$.

\subsection{A dyadic estimate}

We shall need the following dyadic estimate:
\begin{equation}\label{J:230}
  \norm{\Proj_{\xi_0 \cdot \omega \in I_0} \left(
  \Proj_{\R \times T_r(\omega)} u_1^{\gamma,\omega_1} \cdot u_2^{\gamma,\omega_2} \right)}
  \lesssim \left( \frac{r \abs{I_0} L_1L_2}{\gamma} \right)^{1/2} 
  \norm{u_1^{\gamma,\omega_1}}
  \norm{u_2^{\gamma,\omega_2}}.
\end{equation}
By Lemma \ref{B:Lemma2}, we reduce this to the volume estimate
\begin{equation}\label{J:59}
  \abs{E} \lesssim \frac{r\abs{I_1}L_1L_2}{\gamma},
\end{equation}
where $E$ satisfies \eqref{K:12} for some $(\tau_0,\xi_0) \in \R^{1+3}$, with
$$
  R = \bigl\{ \xi \in T_{r}(\omega) \colon \xi \cdot \omega \in I_1,
  \; \theta(e_1,\omega) \le \phi, \; \theta(e_2,\omega_2) \le \gamma \bigr\},
$$
where $I_1$ is some translate of $I_0$, $\phi$ is as in \eqref{J:241} and $e_1,e_2$ are as in \eqref{K:24}. Choose coordinates so that $\omega$ and $\omega_2$ both lie in the $(\xi^1,\xi^2)$-plane, and $\omega = (1,0,0)$. Then $\partial_2 f \sim \gamma$ on $R$, since the sectors $\Gamma_\phi(\omega)$ and $\Gamma_\gamma(\omega_2)$ are separated by an angle comparable to $\gamma \ll 1$, in view of \eqref{J:194} and \eqref{J:244}. Therefore, \eqref{K:16} implies
$$
  \Abs{E} \lesssim \frac{L_1L_2}{\gamma} \Abs{ \left\{ (\xi^1,\xi^3) : \xi^1 \in I_1, \; \abs{\xi^3} \lesssim r \right\} }
  \lesssim \frac{L_1L_2}{\gamma} \abs{I_1} \,r.
$$
This proves \eqref{J:59}, hence \eqref{J:230}.

\subsection{Two general observations}

We shall make use of the following:

\begin{lemma}\label{J:Lemma1}
Given $\omega \in \mathbb S^2$, a compact interval $I_0$ and sets $S_1,S_2 \subset \R^{1+3}$, assume
\begin{equation}\label{J:20}
  \norm{\Proj_{\xi_0 \cdot \omega \in I_0}(u_1u_2)}
  \le \left(A\abs{I_0}\right)^{1/2} \norm{u_1}\norm{u_2}
\end{equation}
for all $u_1,u_2$ satisfying \eqref{J:246}, where $A > 0$ is a constant. Assume further that there exist $d > 0$, $c \in \R$ and a compact interval $J$ such that
\begin{gather}
  \label{J:24}
  d \lesssim \abs{J}, \qquad
  S_2 \subset J \times \R^3,
  \\
  \label{J:28}
  S_2 \subset \left\{ (\tau_2,\xi_2) : -\tau_2 + \xi_2 \cdot \omega = c + O(d) \right\}.
\end{gather}
Then \eqref{J:250} holds with $C^2 \sim A \abs{J}$.
\end{lemma}

\begin{proof} We have
$$
  \text{l.h.s.\eqref{J:20}}
  \lesssim
  \sum_{I_1,I_2}
  \bignorm{u_1^{I_1}}
  \bignorm{u_2^{I_2}}
  \qquad \left( u_j^{I_j} = \Proj_{\xi_j \cdot \omega \in I_j} u_j \right),
$$
where $I_1, I_2$ belong to the almost disjoint cover of $\R$ by translates of $I_0$, and the sum is restricted by the condition $(I_1 + I_2) \cap I_0 \neq \emptyset$, hence each $I_1$ can interact with at most three different $I_2$'s. Thus, sup'ing over $I_1$ and summing $I_2$ using the Cauchy-Schwarz inequality, we get \eqref{J:20}, since the cardinality of the sum over $I_2$ is dominated by $\abs{J}/\abs{I_0}$. To verify the last statement, write
$$
  \xi_2 \cdot \omega = (-\tau_2 + \xi_2 \cdot \omega-c) + \tau_2 + c
$$
and recall \eqref{J:24} and \eqref{J:28}.
\end{proof}

We also need the following.

\begin{lemma}\label{J:Lemma2}
Suppose $\omega \in \mathbb S^2$, $I_0$ is a compact interval, $S_1,S_2 \subset \R^{1+3}$ and $S_1 \subset T_1$, where $T_1 \subset \R^{1+3}$ is an approximate tiling set with the doubling property. If
\begin{equation}\label{J:40}
  \norm{\Proj_{\xi_0 \cdot \omega \in I_0}(u_1\Proj_{T_2}u_2)}
  \le C_0
  \left( \sup_{I_1}\norm{u_1^{I_1}} \right)
  \norm{\Proj_{T_2}u_2}
\end{equation}
for all translates $T_2$ of $T_1$, all translates $I_1$ of $I_0$ and all $u_1,u_2$ as in \eqref{J:246}, then \eqref{J:250} also holds, with a constant $C$ depending on $C_0$ and the size of the overlap of the doubling cover by $T_1$.
\end{lemma} 

\begin{proof} By the definition of an approximate tiling set with the doubling property (section \ref{C}), there exists a lattice $E \subset \R^{1+3}$ such that $\mathcal T_1 = \{ X + T_1 \}_{X \in E}$ is a cover of $\R^{1+3}$ with $O(1)$ overlap, and moreover the corresponding doubling cover $T_1 + \mathcal T_1$ also has $O(1)$ overlap. The Fourier support of $u_1\Proj_{T_2}u_2$ is contained in $T_1 + T_2$, so squaring both sides of \eqref{J:40} and summing over $T_2 \in \mathcal T_1$ yields \eqref{J:250}.
\end{proof}

\subsection{Proof of \eqref{J:220}}

Assume \eqref{J:194} and \eqref{J:240}--\eqref{J:250}. By Remark \ref{C:Rem}, we may assume $L_2 \ll N_2$. We shall apply Lemmas \ref{J:Lemma1} and \ref{J:Lemma2}. By \eqref{J:200} and \eqref{J:230},
\begin{equation}\label{J:58}
  \text{\eqref{J:20} holds with $A \sim \min\left(\dfrac{rL_1L_2}{\gamma},r^2 \Lmin^{12} \right)$.}
\end{equation}
By Lemma \ref{D:Lemma4},
\begin{equation}\label{J:56}
  S_1 \subset T_1 \equiv H_d(\omega) \cap \bigl( \R \times T_{r}(\omega) \bigr)
  \qquad \text{where}
  \qquad
  d = \max(L_1,N_1\phi^2).
\end{equation}
Note that $T_1$ is an approximate tiling set with the doubling property, hence Lemma \ref{J:Lemma2} allows us to replace $S_2$ by $S_2 \cap T_2$ in Lemma \ref{J:Lemma1}, where $T_2$ is an arbitrary translate of $T_1$. Let us fix such a translate:
\begin{equation}\label{J:70}
  T_2 = (\tau_0,\xi_0) + T_1.
\end{equation}
Clearly, \eqref{J:28} holds with $S_2$ replaced by $S_2 \cap T_2$, and with $c = -\tau_0+\xi_0 \cdot \omega$, it only remains to prove the existence of an interval $J$ such that
\begin{equation}\label{J:74:1}
  S_2 \cap T_2 \subset J \times \R^3,
  \qquad
  \abs{J} \sim \max\left(\frac{r}{\gamma}, \frac{\Lmax^{12}}{\gamma^2} \right) \gtrsim d,
\end{equation}
where the very last inequality holds by the definitions of $d$ and $\alpha$ above. Note that $J$ (like $c$) may depend on $(\tau_0,\xi_0)$, which is fixed for the rest of the proof.

Let $(\tau,\xi) \in S_2 \cap T_2$. Then $\xi$ is separated from $\R \omega$ by a distance $\sim N_2\gamma$, and the same is true of $\xi_0$, since $\xi \in \xi_0 + T_{r}(\omega)$ and $r \ll N_2\gamma$. Thus,
\begin{equation}\label{J:80}
  \abs{P_{\omega^\perp} \xi_0} \sim N_2\gamma.
\end{equation}
We also have
\begin{equation}\label{J:82}
  -\tau \pm_2 \abs{\xi} = O(L_2),
  \qquad
  -\tau + \xi \cdot \omega = c + O(d),
\end{equation}
where $c = -\tau_0+\xi_0 \cdot \omega$. Since $L_2 \ll N_2 \sim \abs{\xi_2}$, \eqref{J:82} implies $\pm_2\tau = \abs{\tau} \sim N_2$. But without loss of generality we can take $\pm_2=+$, hence $\tau \sim N_2$.

Choose coordinates $(\xi^1,\xi^2,\xi^3)$ so that $\omega = (1,0,0)$ and $\xi_0 = (\xi_0^1,\xi_0^2,0)$ with $\xi_0^2 \sim N_2\gamma$, by \eqref{J:80}. Each slice $\tau=\text{const}$ of $S_2 \cap T_2$ (where $\tau \sin N_2$) is contained in the intersection of the following three sets (a truncated tube, a thickened sphere and a thickened plane): 
\begin{align*}
  E_1 &= \left\{ \xi \in \R^3 \colon
  \abs{\xi^1} \sim N_2, \; \xi^2 \in [c_1,c_2],
  \; \xi^3 = O(r) \right\},
  \\
  E_2(\tau) &= \left\{ \xi \in \R^3 \colon \abs{\xi} = \tau + O(L_2) \right\},
  \\
  E_3(\tau) &= \left\{ \xi \in \R^3 \colon \xi^1 = \tau + c + O(d) \right\},
\end{align*}
where $0 < c_1 < c_2$ in the definition of $E_1$ satisfy
\begin{equation}\label{J:88}
  c_1,c_2 \sim N_2\gamma, \qquad c_2-c_1 \sim r \ll N_2\gamma.
\end{equation}

Now take a dynamical point of view, thinking of $\tau$ as a time variable. As we increase $\tau$, the sphere expands with unit speed, and the thickened plane moves with unit speed in the $\xi^1$-direction, whereas the tube $E_1$ remains fixed. Since the tube is offset from the $\xi^1$-axis, $E_1(\tau) \cap E_2(\tau)$ is contained in a thickened plane
\begin{equation}\label{J:90}
   E_4(\tau) = \left\{ \xi \colon \xi^1 \in [f(\tau),g(\tau)] \right\},
\end{equation}
moving at a slightly different speed than $E_3(\tau)$, as we show below, causing the two sets to move through each other. Thus, our strategy for proving \eqref{J:74:1} is simply to estimate the length of time for which the two can stay in contact.

We solve $\abs{\xi}^2 = (\xi^1)^2 + \abs{(\xi^2,\xi^3)}^2$ for $\xi^1$, so we want $\xi^1$ to have a definite sign. From the start we can split $S_2$ into two parts, depending on the sign of $\xi \cdot \omega = \xi^1$. Since we chose $\pm_2=+$, the most difficult case is $\xi^1 > 0$, so we assume this (the case $\xi^1 < 0$ is easier: then $E_3,E_4$ will move in opposite directions at approximately unit speed). Thus, $\xi^1 = \sqrt{\abs{\xi}^2 -  \abs{(\xi^2,\xi^3)}^2}$, hence
\begin{equation}\label{J:100}
   E_1(\tau) \cap E_2(\tau) 
   \subset \left\{ \xi \colon
   \abs{\xi} \sim N_2, \;
   \xi^1 \in [x_1(\abs{\xi}),x_2(\abs{\xi})] \right\},
\end{equation}
where $x_1(s) = \sqrt{s^2-\hat c_2^2}$, $x_2(s) = \sqrt{s^2-\hat c_1^2}$, 
for some constants $0 < \hat c_1 < \hat c_2$ satisfying the analogue of \eqref{J:88}. Thus, $x_1(\abs{\xi}) \sim x_2(\abs{\xi}) \sim N_2$ and
\begin{equation}\label{J:110}
  x_1(\abs{\xi}) - x_2(\abs{\xi})
  = \frac{\hat c_2^2-\hat c_1^2}{x_1(\abs{\xi}) + x_2(\abs{\xi})}
  \lesssim \frac{\hat c_2^2-\hat c_1^2}{N_2}
  \sim \frac{(N_2\gamma) r}{N_2} = r \gamma.
\end{equation}
We may interpret $\abs{\xi} = \tau + O(L_2)$ as $\bigabs{-\tau+\abs{\xi}} \le L$ for definiteness, hence
$$
  N_2 \lesssim \tau - L_2
  \le \abs{\xi} \le \tau + L_2
  \lesssim N_2,
$$
where we used $L_2 \ll N_2 \sim \tau$. Thus, since $x_1(s),x_2(s)$ are strictly increasing,
\begin{equation}\label{J:114}
  N_2 \sim x_1(\tau-L_2) \le x_1(\abs{\xi})
  < x_2(\abs{\xi}) \le x_2(\tau+L_2) \sim N_2.
\end{equation}
Since $x_1'(s) \sim x_1'(s) \sim 1$ for $s \sim N_2$, we then get, using also \eqref{J:110},
\begin{equation}\label{J:116}
\begin{aligned}
  &x_2(\tau+L_2) - x_1(\tau-L_2)
  \\
  &\qquad= x_2(\tau+L_2)  - x_2(\abs{\xi})
  + O( r\gamma) + x_1(\abs{\xi}) - x_1(\tau-L_2)
  \\
  &\qquad= O(L_2) + O(r\gamma) + O(L_2).
\end{aligned}
\end{equation}
Moreover, since $x_1'(s) - 1 = \hat c_2^2/[x_1(s)(s+x_1(s))]$, we obtain
\begin{equation}\label{J:118}
  \frac{d}{d\tau}x_1(\tau-L_2)-1
  \sim \frac{(N_2\gamma)^2}{(N_2)^2} 
  \sim \gamma^2.
\end{equation}
Setting
$$
  f(\tau) = x_1(-\tau-L_2),
  \qquad
  g(\tau) = x_2(-\tau+L_2),
$$
we see from \eqref{J:114} and \eqref{J:100} that $E_4(\tau)$ defined by \eqref{J:90} contains $E_1(\tau) \cap E_2(\tau)$. By \eqref{J:116}, the thickness of $E_4(\tau)$ is $O(r\gamma+L_2)$, and by \eqref{J:118}, $E_4(\tau)$ moves with speed $\gamma^2$ relative to $E_3(\tau)$, which has thickness $O(d)$. Therefore, the length of the $\tau$-interval in which the two slabs can intersect is comparable to
$$
  \frac{d+r\gamma+L_2}{\gamma^2}
  \sim \frac{r\gamma+\Lmax^{12}}{\gamma^2},
$$
where we used the fact that $d \lesssim L_1 + r\gamma$, by the definitions of $d$ and $\alpha$.

\section{Proof of Theorem \ref{Z:Thm2}}\label{G:50}

We use coordinates $(x,y,z)$ on $\R^3$. Since $S$ is a surface obtained by revolving a graph $y=f(x)$ about the $x$-axis, surface measure is given by
\begin{equation}\label{G:60}
  d\sigma = f(x) \sqrt{1 + f'(x)^2} \, dx \, d\theta,
\end{equation}
where $\theta$ is the angle of rotation about the $x$-axis. By continuity, we can ignore the case where the normal of the thickened plane is exactly parallel to the $x$-axis. Using also the rotational symmetry of $S$ and $B$ about the $x$-axis, we may therefore assume that $P_\delta$ is the region between the planes
\begin{equation}\label{G:90}
  y = px + q, \quad y = px + q + \frac{\delta}{\cos\beta},
  \quad
  \text{where $p = \tan\beta$ and $-\frac{\pi}{2} < \beta < \frac{\pi}{2}$}.
\end{equation}

In the next two subsections we treat separately the ellipsoid and the hyperboloid, assuming throughout that $a \ge b > 0$ and $b^2/a \lesssim R \lesssim a$.

\subsection{$S$ is an ellipsoid} Then
\begin{equation}\label{G:100}
  f(x) = \frac{b}{a} \sqrt{a^2-x^2},
  \qquad
  d\sigma = \frac{b}{a} \sqrt{a^2 - x^2 + \frac{b^2}{a^2} x^2} \, dx \, d\theta
\end{equation}
and the foci of $S$ are located at $(\pm c,0,0)$, where
\begin{equation}\label{G:98}
  c = \sqrt{a^2 - b^2},
  \qquad
  \text{hence}
  \qquad
  a-c \le \frac{b^2}{a} \lesssim R.
\end{equation}
To make a definite choice, let $B$ be centered at the right focus $(c,0,0)$. We may restrict to the case where the line $y=px+q$ intersects the ellipse $y^2=f(x)^2$ at two points $x_1 < x_2$, which are the zeros of
\begin{equation}\label{G:96}
  Q(x) \equiv \frac{b^2}{a^2}(a^2-x^2)-(px+q)^2
  =
  \frac{b^2(A^2-q^2)}{A^2}-\frac{A^2}{a^2}\left(x+\frac{a^2pq}{A^2}\right)^2,
\end{equation}
where
$$
  A = \sqrt{a^2p^2+b^2}
$$
Assuming $A > q$,
\begin{equation}\label{G:94}
  x_1 = -\frac{a^2pq}{A^2} - d,
  \qquad
  x_2 = -\frac{a^2pq}{A^2} + d,
  \qquad \text{where}
  \qquad
  d = \frac{ab\sqrt{A^2-q^2}}{A^2}.
\end{equation}
By symmetry considerations, we may restrict attention to $x_1 \le x \le x_2$. Intersecting with $B$ further imposes $x \ge c-R$, so we assume $c-R \le x_2$, of course. Thus,
\begin{equation}\label{G:102}
  x \in [x_*,x_2],
  \qquad \text{where} \qquad 
  x_* = \max(x_1,c-R),
\end{equation}
Then $a-x \le a - c + R \lesssim R$, where we used \eqref{G:98}, hence
\begin{equation}\label{G:106}
  \frac{b}{a} \sqrt{a^2 - x^2 + \frac{b^2}{a^2} x^2} \lesssim \frac{b}{\sqrt{a}}\sqrt{R} + \frac{b^2}{a} \lesssim \frac{b}{\sqrt{a}}\sqrt{R},
\end{equation}
giving us control on $d\sigma$.

Now consider a slice $x=\text{const}$ of $S \cap P_\delta$, for some $x \in [x_1,x_2]$. In this slice we see a circle of radius $f(x)$ intersected with the region between the lines
\begin{equation}\label{G:110}
  y = g(x) \equiv px+q,
 \qquad
  y = h(x) \equiv \min\left(px+q+\frac{\delta}{\cos\beta}, f(x) \right),
\end{equation}
hence integration of $\theta$ gives us the following arc length on the unit circle:
\begin{equation}\label{G:130}
  \Delta \theta(x) = 2 \left( \arcsin\frac{h(x)}{f(x)}-\arcsin\frac{g(x)}{f(x)} \right)
  \lesssim \frac{\delta}{\cos\beta}
  \cdot \frac{1}{\sqrt{f(x)^2-g(x)^2}},
\end{equation}
where we used the estimate (proved below)
\begin{equation}\label{G:120}
 \arcsin t - \arcsin s \lesssim \frac{t-s}{\sqrt{1-s^2}}
 \qquad \text{for all $-1 \le s \le t \le 1$}.
\end{equation}
Combining \eqref{G:130} with \eqref{G:102} and \eqref{G:106}, we get
\begin{equation}\label{G:140}
  \sigma\left(S \cap B \cap P_\delta \cap ( [x_*,x_2] \times \R^2) \right)
  \lesssim \frac{b}{\sqrt{a}}\sqrt{R}\frac{\delta}{\cos\beta}
  \int_{x_*}^{x_2} \frac{dx}{\sqrt{Q(x)}}.
\end{equation}
In view of \eqref{G:96} and \eqref{G:94},
\begin{equation}\label{G:150}
  \int_{x_*}^{x_2} \frac{dx}{\sqrt{Q(x)}}
  =
  \frac{a}{A}
  \left(\arcsin 1 - \arcsin \frac{x_*+\frac{a^2pq}{A^2}}{d}\right)
  \lesssim
  \frac{a}{A}\sqrt{\frac{x_2-x_*}{x_2-x_1}}.
\end{equation}
where we used the estimate
$$
  \arcsin 1 - \arcsin s \lesssim \sqrt{1-s} \qquad \text{for $-1 \le s \le 1$}.
$$
This is trivial if $s < 0$; if $0 \le s \le 1$, it follows from \eqref{G:120}.

Since the right side of \eqref{G:150} is no larger than $\pi a/A$,
\begin{equation}\label{G:152}
  \text{l.h.s.\eqref{G:140}}
  \lesssim
  \frac{b\sqrt{a}}{A}\sqrt{R}\frac{\delta}{\cos\beta},
\end{equation}
and this proves the theorem whenever
\begin{equation}\label{G:154}
  \frac{b^2a}{A^2} \lesssim R \cos^2\beta.
\end{equation}
But
\begin{equation}\label{G:156}
  \frac{b^2a}{A^2}
  =
  \frac{b^2a}{a^2p^2+b^2}
  \le
  \min\left(\frac{b^2}{ap^2},a\right)
  \lesssim
  \min\left(\frac{R\cos^2\beta}{\sin^2\beta},a\right),
\end{equation}
implying \eqref{G:154} when $\beta \sim 1$. From now on we can therefore assume $\beta \ll 1$, hence
$$
  \cos\beta \sim 1,
$$
so we can strike this factor from \eqref{G:140} and \eqref{G:154}. Then \eqref{G:156} obviously implies \eqref{G:154} when $R \sim a$, so it remains to consider
\begin{equation}\label{G:160}
  R \ll a, \qquad \text{hence} \qquad c-R \sim a,
\end{equation}
where we used \eqref{G:98} to get the last statement.

Now we split into the cases
\begin{enumerate}
\item\label{G:174} $x_1 \ge c-2R$,
\item\label{G:172} $0 \le x_1 \le c-2R$,
\item\label{G:170} $x_1 < 0$.
\end{enumerate}
Observe that
\begin{equation}\label{G:176}
  \abs{p}
  \ge
  \frac{\Abs{\frac{b}{a}\sqrt{a^2-x_1^2} - \frac{b}{a}\sqrt{a^2-x_2^2}}}{x_2-x_1}
  = \frac{b}{a}
  \frac{\abs{x_1+x_2}}{\sqrt{a^2-x_1^2} + \sqrt{a^2-x_2^2}}.
\end{equation}
If $x_1 \ge 0$, this implies
\begin{equation}\label{G:178}
  \abs{p}
  \ge
  \frac{b}{a}
  \frac{x_1+x_2}{\sqrt{a^2-x_1^2} + \sqrt{a^2-x_2^2}}
  \sim \frac{b}{\sqrt{a(a-x_1)}} \qquad (x_1 \ge 0),
\end{equation}
where we used the fact that $x_2 \sim a$, on account of \eqref{G:160}.

\subsubsection{Case \eqref{G:174}} Then $a - x_1 \le a - c + 2R \lesssim R$, by \eqref{G:98}, hence \eqref{G:178} gives $\abs{p} \gtrsim b/\sqrt{aR}$, implying \eqref{G:154}.

\subsubsection{Case \eqref{G:172}} Then $x_2 - x_1 \ge c-R-x_1 \ge R \gtrsim a-c$, hence
$$
  a - x_1 = a - x_2 + x_2 - x_1 \le a - c + R + x_2 - x_1 \lesssim x_2 - x_1.
$$
Plugging this and $x_2-x_* \le a-c+R \lesssim R$ into \eqref{G:150}, and using \eqref{G:178}, we get
$$
  \text{l.h.s.\eqref{G:140}}
  \lesssim
  \frac{b}{\sqrt{a}}\sqrt{R}\delta \frac{a}{\sqrt{a^2p^2+b^2}} \sqrt{\frac{R}{a-x_1}}
  \lesssim
  \sqrt{R}\delta \frac{a}{a\abs{p}} \sqrt{R}\abs{p}
  \le R\delta.
$$

\subsubsection{Case \eqref{G:170}} Then $x_2-x_1 \ge x_2 \ge c-R \sim a$, so by \eqref{G:150},
$$
  \text{l.h.s.\eqref{G:140}}
  \lesssim
  \frac{b}{\sqrt{a}}\sqrt{R}\delta \frac{a}{\sqrt{a^2p^2+b^2}} \sqrt{\frac{R}{a}}
  \le R\delta.
$$

This concludes the proof for the ellipsoid, except for \eqref{G:120}.

If $0 \le s \le t \le 1$, we write
$$
  \int_s^t \frac{du}{\sqrt{1-u^2}}
  \le \int_s^t \frac{du}{\sqrt{1-u}}
  = 2\left(\sqrt{1-s}-\sqrt{1-t}\,\right) = \frac{2(t-s)}{\sqrt{1-s}+\sqrt{1-t}}.
$$
This also covers $-1 \le s \le t \le 0$, since $\arcsin$ is odd. If $-1 \le s \le 0 \le t \le 1$, we write the left side of \eqref{G:120} as $\arcsin t + \arcsin(-s)$ and use the fact that $\arcsin t \le 2t$ for $0 \le t \le 1$, as follows by the above calculation. This proves \eqref{G:120}.

\subsection{$S$ is a hyperboloid} Then
\begin{equation}\label{G:208}
  f(x) = \frac{b}{a} \sqrt{x^2-a^2},
  \qquad
  d\sigma = \frac{b}{a} \sqrt{x^2 - a^2 + \frac{b^2}{a^2} x^2} \, dx \, d\theta.
\end{equation}
The foci of $S$ are again located at $(\pm c,0,0)$, but now
\begin{equation}\label{G:210}
  c = \sqrt{a^2+b^2},
  \qquad \text{hence} \qquad
  c-a \le \frac{b^2}{a} \lesssim R
\end{equation}
We center $B$ at $(-c,0,0)$, hence we restrict to
\begin{equation}\label{G:212}
  -c-R \le x \le -a.
\end{equation}
Then $\abs{x} = -x \le c+R \lesssim a$, $\abs{x+a} = - (x + a) = -x-c+c-a \lesssim R$ and $\abs{x-a} = a-x \lesssim a$, so to estimate the surface measure we can use
\begin{equation}\label{G:214}
  \frac{b}{a} \sqrt{ x^2 - a^2 + \frac{b^2}{a^2}x^2}
  \lesssim \frac{b}{\sqrt{a}}\sqrt{R} + \frac{b^2}{a}
  \lesssim \frac{b}{\sqrt{a}}\sqrt{R}.
\end{equation}
There are two cases to consider:
\begin{enumerate}
  \item\label{G:220} $\abs{p} \le b/a$,
  \item\label{G:222} $\abs{p} > b/a$,
\end{enumerate}
where $b/a$ is, of course, the asymptotic slope of $y=f(x)$.

\subsubsection{Case \eqref{G:220}} Then
\begin{equation}\label{G:230}
  Q(x) \equiv \frac{b^2}{a^2}(x^2-a^2) - (px+q)^2
  =
  \frac{A^2}{a^2}\left(x - \frac{a^2pq}{A^2}\right)^2
  - \frac{b^2(q^2+A^2)}{A^2},
\end{equation}
where
$$
  A = \sqrt{b^2-a^2p^2},
$$
and $Q(x)$ has exactly one zero in the interval $(-\infty,-a]$, namely $x_2$ given by
\begin{equation}\label{G:238}
  -x_2 = -\frac{a^2pq}{A^2} + \frac{ab\sqrt{q^2+A^2}}{A^2}
  \ge \frac{ab}{A^2} \left( \sqrt{q^2+A^2} - \abs{q} \right)
  \sim \frac{ab}{\sqrt{q^2+A^2}},
\end{equation}
where we used $\abs{p} \le b/a$. We may restrict to
\begin{equation}\label{G:240}
  x \in [x_*,x_2], \qquad \text{where} \qquad x_* = - c - R,
\end{equation}
and we assume $x_* \le x_2$, of course, hence \eqref{G:238} implies
\begin{equation}\label{G:242}
  \frac{ab}{\sqrt{q^2+A^2}} \lesssim c+R \lesssim a+R \lesssim a
  \qquad \text{hence}
  \qquad
  \sqrt{q^2+A^2} \gtrsim b.
\end{equation}
Since \eqref{G:110} and \eqref{G:130} still apply, then using \eqref{G:208} and \eqref{G:214} we get \eqref{G:140}, and we have the estimate (proved below)
\begin{equation}\label{G:262}
  \phi(x) \equiv \int_{x}^{x_2} \frac{dt}{\sqrt{Q(t)}}
  \lesssim \sqrt{\frac{a(x_2-x)}{b\sqrt{q^2+A^2}}} 
  \qquad (x \le x_2).
\end{equation}
Plugging this into \eqref{G:140}, and using $x_2-x_* \le -a+c+R \lesssim R$ and \eqref{G:242},
$$
  \text{l.h.s.\eqref{G:140}}
  \lesssim \frac{b}{\sqrt{a}}\sqrt{R}\delta
  \sqrt{\frac{a(x_2-x_*)}{b\sqrt{q^2+A^2}}}
  \lesssim \frac{b}{\sqrt{a}}\sqrt{R}\delta
  \sqrt{\frac{aR}{b^2}} = R\delta,
$$
where we used also the fact that $\cos\beta \sim 1$, since $\abs{p} \le b/a \le 1$. 

This concludes the proof for case \eqref{G:220}, up to the claim \eqref{G:262}, but in view of \eqref{G:230} and \eqref{G:238}, $\phi(x) = \frac{a}{A} \psi(d+x_2-x)$, where $d = \frac{ab\sqrt{q^2+A^2}}{A^2}$ and
$$
  \psi(u) = \int_d^u \frac{ds}{\sqrt{s^2-d^2}}
  \le \frac{1}{\sqrt{d}} \int_d^u \frac{ds}{\sqrt{s-d}}
  =  \frac{2\sqrt{u-d}}{\sqrt{d}} \qquad (u \ge d),
$$
implying \eqref{G:262}.

\subsubsection{Case \eqref{G:222}} Then we have (this should be compared with \eqref{G:96})
\begin{equation}\label{G:280}
  Q(x) \equiv \frac{b^2}{a^2}(x^2-a^2) - (px+q)^2
  =
  \frac{b^2(q^2-A^2)}{A^2}
  - \frac{A^2}{a^2}\left(x + \frac{a^2pq}{A^2}\right)^2,
\end{equation}
where
$$
  A = \sqrt{a^2p^2-b^2}.
$$
We may restrict attention to the case where $Q(x)$ has two zeros $x_1 < x_2$. This happens when $A^2 < q^2$, and then $x_1,x_2$ are given by \eqref{G:94}, but now with
\begin{equation}\label{G:290}
  d = \frac{ab\sqrt{q^2-A^2}}{A^2}.
\end{equation}
The zeros are either both in the interval $(-\infty,-a]$ or both in $[a,\infty)$, and we assume of course the former, which happens when $p,q$ have the same sign. Since we are intersecting with the ball $B$, we further assume $x_2 \ge -c-R$, hence
\begin{equation}\label{G:300}
  x \in [x_*,x_2], \qquad \text{where} \qquad x_* = \max(x_1,-c-R).
\end{equation}
Now \eqref{G:140}--\eqref{G:152} are valid, so the theorem follows whenever \eqref{G:154} holds. It certainly holds if $p \ge 2$, say, since then $A^2 \sim a^2p^2$. This takes care of the case $\cos\beta \ll 1$, so from now on we may assume
$$
  \cos\beta \sim 1,
$$
hence we can strike this factor from \eqref{G:140}, \eqref{G:152} and \eqref{G:154}.

Observe that $\abs{p}$ must be at least as large as the absolute value of the slope of $y=f(x)$ at $x=x_1$, hence
$$
  a\abs{p} \pm b \ge b\left(\frac{\abs{x_1}}{\sqrt{x_1^2-a^2}} \pm 1 \right),
$$
implying
\begin{equation}\label{G:304}
  A^2 = (a\abs{p}+b)(a\abs{p}-b)
  \ge
  b^2
  \left(
  \frac{x_1^2}{x_1^2-a^2}
  -1 \right)
  =
  \frac{a^2b^2}{x_1^2-a^2}.
\end{equation}

We now split into two subcases:
{
\renewcommand{\theenumi}{\alph{enumi}}
\begin{enumerate}
  \item\label{G:310} $x_1 \ge -c-2R$,
  \item\label{G:312} $x_1 < -c-2R$.
\end{enumerate}
}

In subcase \eqref{G:310}, $\abs{x_1} \sim a$ and $\abs{x_1}-a \lesssim R$, so \eqref{G:304} implies $A^2 \gtrsim ab^2/R$, hence \eqref{G:154} holds.

Now assume subcase \eqref{G:312}. Then we claim that
\begin{equation}\label{G:350}
  \abs{x_1}-a \sim x_2 - x_1.
\end{equation}
Indeed, $x_2-x_1 \le -a-x_1= \abs{x_1}-a$, and
$$
  -x_2-a \le c+R - a = c-a+R \lesssim R = - c - R - (-c-2R) \le x_2 - x_1,
$$
hence $-x_1-a = x_2-x_1 -x_2-a \lesssim x_2-x_1$, proving \eqref{G:350}. By \eqref{G:140} and \eqref{G:150}, and using $x_2-x_* \le -a+c+R \lesssim R$ and \eqref{G:350}, we get
$$
  \text{l.h.s.\eqref{G:140}}
  \lesssim \frac{b}{\sqrt{a}}\sqrt{R}\delta
  \frac{a}{A}
  \sqrt{\frac{x_2-x_*}{x_2-x_1}}
  \lesssim
  R\delta \frac{\sqrt{a}b}{A\sqrt{\abs{x_1}-a}}
$$
hence it is enough to show
\begin{equation}\label{G:360}
  A^2 \gtrsim \frac{ab^2}{\abs{x_1}-a}.
\end{equation}
This follows from \eqref{G:304} if $\abs{x_1} \sim a$, so it remains to consider
\begin{equation}\label{G:362}
  \abs{x_1} \gg a, \qquad \text{hence} \qquad x_1^2-a^2 \sim x_1^2.
\end{equation}
Then instead of \eqref{G:304} we use the analogue of \eqref{G:176}, which implies
\begin{align*}
  A^2 \ge b(a\abs{p}-b)
  &\ge
  b^2 \left( \frac{-x_1-x_2}{\sqrt{x_1^2-a^2}+\sqrt{x_2^2-a^2}}
  - 1 \right)
  \\
  &=
  b^2 \frac{\abs{x_1}-\sqrt{x_1^2-a^2}+\abs{x_2}-\sqrt{x_2^2-a^2}}{\sqrt{x_1^2-a^2}+\sqrt{x_2^2-a^2}}
  \\
  &=
  b^2 \frac{\frac{a^2}{\abs{x_1}+\sqrt{x_1^2-a^2}}
  + \frac{a^2}{\abs{x_2}+\sqrt{x_2^2-a^2}}}{\sqrt{x_1^2-a^2}+\sqrt{x_2^2-a^2}}
  \sim \frac{ab^2}{\abs{x_1}},
\end{align*}
where at the end we used \eqref{G:362} and the fact that $\abs{x_2} \sim a$. Thus, \eqref{G:360} again holds.

This concludes the proof of Theorem \ref{Z:Thm2}.

\section{Proofs of lemmas}\label{G}

\subsection{Proof of Lemma \ref{C:Lemma1}}

Choose coordinates so that $\xi_0 = (\abs{\xi_0},0,0) \neq 0$. Write $\xi = (\xi^1,\xi')$, where $\xi'=(\xi^2,\xi^3)$. Then $\xi \in S_\delta(r) \cap (\xi_0 + S_\Delta(R) )$ if and only if
$$
  (r-\delta)^2 < (\xi^1)^2 + \abs{\xi'}^2 < (r+\delta)^2
$$
and
$$
  (R-\Delta)^2 < (\xi^1-\xi_0^1)^2 + \abs{\xi'}^2 < (R+\Delta)^2.
$$
Subtracting these inequalities, we find that
\begin{multline*}
  \xi \in S_\delta(r) \cap ( \xi_0 + S_\Delta(R) )
  \implies \xi^1 \in (a,b),
  \\
  \text{where} \quad
  \left\{
  \begin{aligned}
  a &= \frac{1}{2\abs{\xi_0}} \left(\abs{\xi_0}^2+r^2-R^2+\delta^2-\Delta^2-2(r\delta+R\Delta)\right),
  \\
  b &= \frac{1}{2\abs{\xi_0}} \left(\abs{\xi_0}^2+r^2-R^2+\delta^2-\Delta^2+2(r\delta+R\Delta)\right).
  \end{aligned}
\right.
\end{multline*}
But $b-a = 2(r\delta+R\Delta)/\abs{\xi_0}$, so to complete the proof we can apply the following lemma. To be precise, by symmetry we may assume without loss of generality that $r\delta \le R\Delta$, and we then apply the following lemma with $(\rho,\varepsilon) = (r,\delta)$, thus completing the proof of Lemma \ref{C:Lemma1}.

\begin{lemma}\label{M:Lemma3} Let $a,b \in \R$ with $a < b$, and let $0 < \varepsilon \ll \rho$. Then
$$
  \Abs{S_\varepsilon(\rho) \cap \left\{ \xi \colon a < \xi^1 < b \right\}}
  \lesssim \rho \varepsilon (b-a).
$$
\end{lemma}

\begin{proof} Without loss of generality assume $0 \le a < b \le \rho+\varepsilon$. We split into the cases (i) $b \le \rho-\varepsilon$ and (ii) $\rho-\varepsilon < b \le \rho+\varepsilon$. In case (i) we calculate the volume as
$$
  \int_a^b \pi \bigl( (\rho^*)^2 - (\rho_*)^2 \bigr) \, d\xi^1 = 4\pi\rho\varepsilon(b-a),
$$
where $\rho^* = \bigl((\rho+\varepsilon)^2-(\xi^1)^2)^{1/2}$ and $\rho_* = \bigl((\rho-\varepsilon)^2-(\xi^1)^2\bigr)^{1/2}$.

Next, assume (ii). Then we can set $a=\rho-\varepsilon$, since the interval $a \le \xi^1 \le \rho-\varepsilon$ is covered by case (i). Therefore, the volume is
$$
  \int_a^b \pi (\rho^*)^2 \, d\xi^1 \le \int_a^b \pi \bigl( (\rho+\varepsilon)^2 - (\rho-\varepsilon)^2 \bigr) \, d\xi^1 = 4\pi \rho\varepsilon (b-a).
$$
\end{proof}

\subsection{Proof of Lemma \ref{I:Lemma1}}

First observe that whenever \eqref{I:16} holds, then so does \eqref{I:10}, since $\xi_0=\xi_1+\xi_2$, hence $1 \lesssim \max(\abs{\xi_1},\abs{\xi_2}) / \abs{\xi_0}$ by the triangle inequality.

It suffices to consider $(\pm_1,\pm_2) = (+,+)$ and $(+,-)$.

If $(\pm_1,\pm_2) = (+,+)$, we use \eqref{A:138} to write
\begin{equation}\label{I:20}
  \hypwt_0-\hypwt_1-\hypwt_2
  =
  ( - \tau_0 \pm_0 \abs{\xi_0} )
  - ( - \tau_1 + \abs{\xi_1} )
  - ( - \tau_2 + \abs{\xi_2} )
  = \pm_0\abs{\xi_0} - \abs{\xi_1} - \abs{\xi_2}.
\end{equation}
The absolute value of the right side is bounded below by $\abs{\xi_1} + \abs{\xi_2} - \abs{\xi_0}$, so \eqref{I:2} gives \eqref{I:10}. Moreover, if $\abs{\xi_0} \ll \abs{\xi_1} \sim \abs{\xi_2}$, then \eqref{B:4} shows that $\theta_{12} \sim 1$.

If $(\pm_1,\pm_2) = (+,-)$, we write
\begin{equation}\label{I:22}
  \hypwt_0-\hypwt_1-\hypwt_2
  = ( - \tau_0 \pm_0 \abs{\xi_0} )
  - ( - \tau_1 + \abs{\xi_1} )
  - ( - \tau_2 - \abs{\xi_2} )
  = \pm_0\abs{\xi_0} - \abs{\xi_1} + \abs{\xi_2}.
\end{equation}
The absolute value of the right side is bounded below by $\abs{\xi_0} - \bigabs{\abs{\xi_1} - \abs{\xi_2}}$, so \eqref{I:16} follows from \eqref{I:4}.

At this point, we have proved \eqref{I:10}, we have proved that \eqref{I:12} implies $\theta_{12} \sim 1$, and we have proved that \eqref{I:16} holds when $\pm_1 \neq \pm_2$. Now observe that if \eqref{I:12} does not hold, then either $\pm_1 \neq \pm_2$, in which case we already know that \eqref{I:16} holds, or $\abs{\xi_0} \sim \max(\abs{\xi_1},\abs{\xi_2})$, in which case \eqref{I:16} follows from \eqref{I:10}.

It remains to prove \eqref{G:12} under the assumptions \eqref{G:8}, \eqref{G:10}. If $(\pm_1,\pm_2) = (+,+)$, then by \eqref{I:20} and \eqref{G:8},
$$
  \bigabs{-\tau_0\pm_0\abs{\xi_0}}
  = \bigabs{\abs{\tau_0} - \abs{\xi_0}} \sim
  \abs{\xi_1}+\abs{\xi_2} - (\pm_0 \abs{\xi_0}),
$$
but the sign $\pm_0$ was chosen so that left hand side is minimal, and from the form of the right hand side we then see that necessarily $\pm_0 = +$, so \eqref{G:12} follows from \eqref{I:2}. The proof in the case $(\pm_1,\pm_2) = (+,-)$ is quite similar; we omit the details.

This completes the proof of the lemma.

\subsection{Proof of Lemma \ref{K:Lemma1}}

The boundary of $D$ is a circle of radius $\sin\theta$ in a plane which makes an angle $\pi/2-\beta$ with the $(\xi^1,\xi^2)$-plane, hence it projects onto an ellipse with semiaxes $a=\sin\theta$ and $b=\sin\theta\sin\beta$, in the $\xi^2$- and $\xi^1$-directions respectively. The center of the ellipse is on the $\xi^1$-axis, with coordinate $\xi^1 = \cos\beta\cos\theta$. The points of intersection between the ellipse and the line $\xi^2=x$, for a given $0 \le x \le \sin\theta$, are therefore given by $\xi^1 = \cos\beta\cos\theta \pm y$, where $y > 0$ is the solution of $x^2/a^2 + y^2/b^2 = 1$, hence $y$ is as in the lemma. This proves part \eqref{K:Lemma1a}.

Now assume $\beta < \theta \le \pi/2$, so $D$ dips below $\mathbb S^2_+$. By trigonometry in the $(\xi^1,\xi^3)$-plane, the disk intersects the boundary of $\mathbb S^2_+$ at $\xi^1$-coordinate $\cos\theta/\cos\beta$, and since the points of intersection lie on the unit circle in the $(\xi^1,\xi^2)$-plane, we get the $\xi^2$-coordinate $x$ as in part \eqref{K:Lemma1b} of the lemma. The rest of the proof then goes as in part \eqref{K:Lemma1a}, but with the crucial difference that the upper bound on $\xi^1$ that we had in $R$ is not present in $R'$, since that upper bound corresponds to a part of the disc which is now already below the equator.

\bibliographystyle{amsplain} 
\bibliography{bibliography}

\end{document}